\documentclass{birkjour}
 \newtheorem{thm}{Theorem}[section]
 \newtheorem{cor}[thm]{Corollary}
 \newtheorem{lem}[thm]{Lemma}
 \newtheorem{prop}[thm]{Proposition}
 \theoremstyle{definition}
 
 \theoremstyle{remark}
 \newtheorem{rem}[thm]{Remark}
 \newtheorem{rems}[thm]{Remarks}
 
 \numberwithin{equation}{section}

\usepackage{hyperref}

\usepackage[cmtip,all]{xy}
\usepackage{enumerate}
\usepackage{amsmath,amssymb,amsxtra,amsbsy,framed,color}
\usepackage[mathscr]{eucal}
\usepackage{array}
\allowdisplaybreaks

\newcommand{\hs}{\hskip-0.1em}

\newcommand{\vanish}[1]{\relax}       
\newcommand{\emdf}{\bf}               
\def\qedsymbol{\hbox to 1ex{\llap{\rule{0.25pt}{1ex}}\rlap{\rule{1ex}{0.25pt}}\lower0.25pt\rlap{\raise1ex\rlap{\rule{1ex}{0.25pt}}}\hskip1ex\llap{\rule{0.25pt}{1ex}}}}
\def\rlqed{\rlap{\rule{\hsize}{0pt}\kern-1ex\kern-1em\qed}} 

\newcommand{\fin}{\mathrm{fin}}
\newcommand{\Ext}{\mathrm{Ext}}
\newcommand{\Res}{\mathrm{Res}}

\newcounter{aufzi}
\newenvironment{aufzi}{\begin{list}{ {\upshape\alph{aufzi})}}{
        \usecounter{aufzi}
        \topsep1ex
        \parsep0cm
        \itemsep0.8ex
        \leftmargin0.7cm  
        \labelwidth0.5cm
        \labelsep0.3cm
}}
{\end{list}}

\newcounter{aufzii}
\newenvironment{aufzii}{\begin{list}{\hfill {\upshape
(\roman{aufzii})}}{
        \usecounter{aufzii}
        \topsep1ex
        \parsep0cm
        \itemsep1ex
        \leftmargin1cm
        \labelwidth0.5cm
        \labelsep0.3cm
}}
{\end{list}}

\newcounter{aufziii}
\newenvironment{aufziii}{\begin{list}{ {\upshape\arabic{aufziii})}}{
        \usecounter{aufziii}
        \topsep1ex
        \parsep0cm
        \itemsep0.8ex
        \leftmargin0.7cm  
        \labelwidth0.5cm
        \labelsep0.3cm
}}
{\end{list}}



\newcommand{\calB}{\mathcal{B}}

\newcommand{\calL}{\mathcal{L}}
\newcommand{\calM}{\mathcal{M}}

\def\bff{\mathbf{f}}

\def\ud{\mathrm{d}}   
\def\ue{\mathrm{e}}

\def\ui{\mathrm{i}}

\newcommand{\upi}{\pi}      

\renewcommand{\ud}{\mathrm{d}}    
\renewcommand{\ue}{\mathrm{e}}    
\renewcommand{\ui}{\mathrm{i}}    
\newcommand{\car}{\mathbf{1}}           

\newcommand{\veps}{\varepsilon}
\newcommand{\vphi}{\varphi}

\newcommand{\T}{\mathbb{T}}
\newcommand{\R}{\mathbb{R}}     
\newcommand{\N}{\mathbb{N}}

\newcommand{\C}{\mathbb{C}}

\newcommand{\torus}{\T}          

\newcommand{\st}{\,:\,}     

\newcommand{\set}[1]{\hs\left[\,#1\,\right]}   

\newcommand{\ohne}{\setminus}
 
\newcommand{\Bigcap}[2][\relax]{%
 \ifx#1\relax \bigcap_{#2}
 \else \bigcap^{#1}_{#2}
 \fi}
\newcommand{\Bigcup}[2][\relax]{%
 \ifx#1\relax \bigcup_{#2}
 \else \bigcup^{#1}_{#2}
 \fi}

\newcommand{\Dan}{\,\Longrightarrow\,}

\newcommand{\nach}{\circ}

\DeclareMathOperator{\dom}{dom}
\DeclareMathOperator{\ran}{ran}

\newcommand{\Inf}{\,\wedge\,}

\newcommand{\conj}[1]{\overline{#1}}   
\DeclareMathOperator{\re}{Re}          
\DeclareMathOperator{\im}{Im}          
\newcommand{\abs}[1]{\left\vert#1\right\vert}   
\def\Re{\re}      
\def\Im{\im}

\newcommand{\Ball}{\mathrm{B}}          

\newcommand{\tensor}{\otimes}

\def\fact#1#2{#1/#2}
\def\tfact#1#2{#1/#2}

\def\fact#1#2{{\raise0.2em\hbox{$#1$}\kern-0.2em/\kern-0.1em\lower0.2em\hbox{$#2$}}}
\def\tfact#1#2{{\raise0.1em\hbox{\small$#1$}\kern-0.1em/\kern-0.1em\lower0.1em\hbox{\small$#2$}}}

\newcommand{\norm}[2][\relax]{
   \ifx#1\relax \ensuremath{\left\Vert#2\right\Vert}
   \else \ensuremath{\left\Vert#2\right\Vert_{#1}}
   \fi}

\newcommand{\Bnorm}[2][\relax]{
   \ifx#1\relax \ensuremath{\Bigl\Vert#2\Bigr\Vert}
   \else \ensuremath{\Bigl\Vert#2\Bigr\Vert_{#1}}
   \fi}
\def\Id{\mathrm{I}}           

\newcommand{\BL}{\calL}                

\makeatletter
\newcommand{\tdprod}[2]{\ensuremath{%
  \setbox0=\hbox{\ensuremath{\langle#1,#2 \rangle}}
  \dimen@\ht0
  \advance\dimen@ by \dp0 \langle#1\rule[-\dp0]{0pt}{\dimen@}\,,#2\hspace{1pt}\rangle}}
\newcommand{\dprod}[2]{\ensuremath{%
  \setbox0=\hbox{\ensuremath{\left\langle#1,#2\right\rangle}}
  \dimen@\ht0
  \advance\dimen@ by \dp0 \left\langle\left.#1\rule[-\dp0]{0pt}{\dimen@},\right.#2\hspace{1pt}\right\rangle}}

\newcommand{\bdprod}[2]{\ensuremath{%
  \setbox0=\hbox{\ensuremath{\bigl\langle#1,#2\bigr\rangle}}
  \dimen@\ht0
  \advance\dimen@ by \dp0 \bigl\langle#1\bigl|\rule[-\dp0]{0pt}{\dimen@}\bigr.#2\hspace{1pt}\bigr\rangle}}
\newcommand{\Bdprod}[2]{\ensuremath{%
  \setbox0=\hbox{\ensuremath{\Bigl\langle#1,#2\Bigr\rangle}}
  \dimen@\ht0
  \advance\dimen@ by \dp0 \Bigl\langle#1\Bigl|\rule[-\dp0]{0pt}{\dimen@}\Bigr.#2\hspace{1pt}\Bigr\rangle}}
\newcommand{\tsprod}[2]{\ensuremath{%
  \setbox0=\hbox{\ensuremath{(#1,#2)}}
  \dimen@\ht0
  \advance\dimen@ by \dp0 (#1\rule[-\dp0]{0pt}{\dimen@}\,|#2\hspace{1pt})}}
\newcommand{\sprod}[2]{\ensuremath{%
  \setbox0=\hbox{\ensuremath{\left(#1,#2\right)}}
  \dimen@\ht0
  \advance\dimen@ by \dp0 \left(\left.#1\rule[-\dp0]{0pt}{\dimen@}\,\right|#2\hspace{1pt}\right)}}
\newcommand{\bsprod}[2]{\ensuremath{%
  \setbox0=\hbox{\ensuremath{\bigl(#1,#2\bigr)}}
  \dimen@\ht0
  \advance\dimen@ by \dp0 \bigl(#1\bigl|\rule[-\dp0]{0pt}{\dimen@}\bigr.#2\hspace{1pt}\bigr)}}
\newcommand{\Bsprod}[2]{\ensuremath{%
  \setbox0=\hbox{\ensuremath{\Bigl(#1,#2\Bigr)}}
  \dimen@\ht0
  \advance\dimen@ by \dp0 \Bigl(#1\Bigl|\rule[-\dp0]{0pt}{\dimen@}\Bigr.#2\hspace{1pt}\Bigr)}}
\makeatother

\newcommand{\Ce}{\mathrm{C}}
\newcommand{\Cc}{\mathrm{C}_\mathrm{c}}

\newcommand{\Co}[1][\relax]{%
\mathrm{C}_\mathrm{0}^\mathrm{#1}}

\newcommand{\Ell}[2][\relax]{
   \ifx#1\relax \mathrm{L}^{\mathrm{#2}}
   \else \mathrm{L}^{\mathrm{#2}}_{\mathrm{#1}}
   \fi}
\renewcommand{\Ell}[2][\relax]{
   \ifx#1\relax \mathrm{L}^{\!#2}
   \else \mathrm{L}^{\!#2}_{\mathrm{#1}}
   \fi}

\newcommand{\Wee}[2][\relax]{
   \ifx#1\relax \mathrm{W}^{\mathrm{#2}}
   \else \mathrm{W}^{\mathrm{#2}}_{\mathrm{#1}}
   \fi}
\newcommand{\Har}[2][\relax]{
   \ifx#1\relax \mathsf{H}^{\mathsf{#2}}
   \else   \mathsf{H}^{\mathsf{#2}}_{\mathrm{#1}}
   \fi}
\newcommand{\eM}{\mathrm{M}}     

\def\prX{\mathrm X}    
\def\prY{\mathrm Y}

\def\prZ{\mathrm Z}

\def\rlqed{\rlap{\rule{\hsize}{0pt}\kern-1ex\kern-1em\qed}}

\makeatletter

\def\maketag@@@@@#1{\llap{\hbox to\hsize{\m@th\normalfont#1}}%
\gdef\tagform@##1{\maketag@@@{(\ignorespaces##1\unskip\@@italiccorr)}}}

\def\eqtext#1{\gdef\tagform@##1{\maketag@@@@@{\ignorespaces##1\unskip\@@italiccorr\hfill}}\tag{#1}}%
\def\reqtext#1{\gdef\tagform@##1{\maketag@@@@@{\hfill\ignorespaces##1\unskip\@@italiccorr}}\tag{#1}}%
\def\leqtext#1{\gdef\tagform@##1{\maketag@@@@@{\ignorespaces##1\unskip\@@italiccorr}}\tag{#1}}%
\makeatother

\newcommand{\Eff}{F}
\newcommand{\Gee}{G}
\newcommand{\Dragicevic}{Dragi\v{c}evi\'c}
\newcommand{\linabs}[1]{\abs{#1}\!}

\begin{document}

%
%
%
%
%
%
%
%
%

\title[Form Inequalities]{Form Inequalities for  
Symmetric Contraction Semigroups}

\author[Haase]{Markus Haase}

\address{%
Christian-Albrechts Universit\"at zu Kiel\\
Mathematisches Seminar\\
Ludewig-Meyn-Stra\ss e 4, 
24098 Kiel, Germany}

\email{haase@math.uni-kiel.de}

\thanks{Part of this work was 
supported by the Marsden Fund Council from Government funding, 
administered by the Royal Society of New Zealand.}

\vanish{
\author{A Second Author}
\address{The address of\br
the second author\br
sitting somewhere\br
in the world}
\email{dont@know.who.knows}
}

\subjclass{47A60, 47D06, 47D07, 47A07}

\keywords{symmetric contraction semigroup, diffusion semigroup, 
sector of analyticity, 
Stone model, integral bilinear forms, 
tensor products}

\date{\today}

\begin{abstract}
Consider --- for the generator ${-}A$ of a symmetric contraction semigroup
over some measure space $\prX$, $1\le p < \infty$, $q$ the dual exponent
and  given measurable functions
$\Eff_j,\: \Gee_j : \C^d \to \C$ --- the  statement:
\[ \Re \sum_{j=1}^m \int_\prX A\Eff_j(\bff) \cdot
\Gee_j(\bff) \,\,\ge \,\,0
\]
{\em for  all  $\C^d$-valued measurable functions $\bff$ on $\prX$
such that $\Eff_j(\bff) \in \dom(A_p)$ and $\Gee_j(\bff)
\in \Ell{q}(\prX)$ for all $j$.}

It is shown that this statement is valid in general if it is valid
for  $\prX$ being a two-point Bernoulli $(\frac{1}{2}, \frac{1}{2})$-space 
and $A$ being of a special form. 
As a consequence we obtain a new proof for 
the optimal angle of $\Ell{p}$-analyticity for such semigroups, which 
is essentially the same as in the well-known sub-Markovian case.

The proof of the main theorem is a combination of well-known reduction 
techniques and some representation results about operators 
on  $\Ce(K)$-spaces. One focus of the paper lies on presenting
these auxiliary techniques and results in great detail. 
\end{abstract}

\maketitle

Second Version, {\em 29 August, 2015}.

\section{Introduction}

In the recent preprint \cite{CarbDrag13Pre}, A.{} Carbonaro and
O.{} \Dragicevic\ consider
symmetric contraction semigroups $(S_t)_{t\ge 0}$ over some 
measure space $\prX= (X, \Sigma,\mu)$ and prove so-called
spectral multiplier results (= functional calculus estimates)  
for $A_p$, where $-A_p$ is the generator of $(S_t)_{t\ge 0}$ 
on $\Ell{p}(\prX)$, $1\le p < \infty$.

Their proof consists of three major steps. In the first one,
the authors show how to generate functional calculus estimates for the
operator $A= A_p$ from  {\em form inequalities}
of the type
\begin{equation}\label{int.eq:ineq} 
\sum_{j=1}^m \re \int_X [A \Eff_j(f_1, \dots, f_d)] \cdot \Gee_j(f_1, \dots, f_d)\, \ud{\mu} \ge 0,
\end{equation}
where $\Eff_j$ and $\Gee_j$ are  
measurable functions $\C^d \to \C$ with certain properties
and $(f_1 \dots, f_d)$ varies over a suitable
subset of measurable functions on $\prX$. This first step is based on 
the so called {\em heat-flow method}.
In the second step, the authors show how to find 
functions $F_j$ and $G_j$ with the desired properties by employing
a so-called {\em Bellman function}.
Their third step consists in establishing the inequality \eqref{int.eq:ineq}
by reducing the problem to the case that 
$A = \Id - E_\lambda$ on $\C^2$, where
\[ E_\lambda = \begin{pmatrix} 0 & \conj{\lambda}  \\ \lambda &  0\end{pmatrix}, 
\qquad (\lambda\in \torus).
\]
The underlying reduction procedure is actually well-known in the literature,
but has been used mainly for symmetric {\em sub-Markovian} semigroups, i.e.,
under the additional assumption that all $S_t\ge 0$. Here, the last
step becomes considerably simpler, since then one need only consider the cases
$A= \Id -  E_1$ and $A= \Id$.

\medskip
\noindent
One intention with the present paper is to look
more carefully at
the employed reduction techniques (Section \ref{s.red})
and prove a general theorem
(Theorem \ref{mr.t:main})
that puts the abovementioned ``third step'' on a formal basis.
Where the authors of \cite{CarbDrag13Pre}
confine their arguments to their specific case of Bellman functions,
here we treat general functions $F_j$ and $G_j$ and hence
pave the way for further applications. 

It turns out that the heart of the matter are results about
representing bilinear forms $(f,g) \mapsto 
\int_L Tf\cdot g \, \ud{\mu}$
as integrals over product spaces like
\[ \int_{L} Tf\cdot g\, \ud{\mu} = \int_{K \times L} f(x)g(y)\, \ud{\mu_T(x,y)}.
\]
(Here, $K$ and $L$ are compact spaces, $\mu$ is
a positive regular Borel measure on $L$ and $T: \Ce(K) \to \Ell{1}(L,\mu)$ is
a linear operator.) These results go back to Grothendieck's 
work on tensor products and ``integral'' bilinear forms \cite{GrothendieckPT}.
They are ``well-known'' in the sense that they could ---
on a careful reading --- be obtained  from standard texts on tensor products 
and Banach lattices, such as \cite[Chap.{} IV]{Schaefer1974}. However,
it seems that the communities of those people who are familiar with these
facts in their abstract form  
and those who would like to apply them to more concrete situations 
are almost disjoint. Our exposition, forming the contents of Section 
\ref{s.ot},  can thus be viewed as an attempt to 
increase the intersection of these two communities.

After this excursion into abstract operator theory, 
in Section \ref{s.prf} we turn back
to the proof of Theorem \ref{mr.t:main}. Then, 
 as an application, we consider the question about the
optimal {\em angle of analyticity} on $\Ell{p}$ of
a symmetric contraction semigroup $(S_t)_{t\ge 0}$. For the sub-Markovian
case this question has been answered long ago, in fact, by the very methods
which we just mentioned and which  form the core content of this paper. 
The general symmetric case has only recently been settled by 
Kriegler in \cite{Kriegler2011}.
Kriegler's proof rests on arguments from non-commutative
operator theory, but Carbonaro and \Dragicevic\ show in \cite{CarbDrag13Pre}
that the result can also be derived as a corollary from their results involving
Bellman-functions. We shall point out in Section \ref{s.soa} below
that the Bellman function of Carbonaro and \Dragicevic\ 
is not really needed here, and 
that one can prove the general case by essentially the same
arguments as used in the sub-Markovian case.


\bigskip

\paragraph{Terminology and Notation} 
In this paper, $\prX:= (X, \Sigma,\mu)$ denotes  a general measure space.
(Sometimes we shall suppose in addition that $\mu$ is a finite measure, but
we shall always make this explicit.) Integration with respect to $\mu$
is abbreviated by 
\[ \int_\prX  f := \int_X f \, \ud{\mu}
\]
whenever it is convenient. 
The corresponding $\Ell{p}$-space for $0 < p\le \infty$ is denoted by 
$\Ell{p}(\prX)$, but if the underlying measure space is
understood, we shall simply write $\Ell{p}$.
Whenever $1 \le p \le \infty$ is fixed we 
denote by $q$ the {\em dual exponent}, i.e., 
the unique number $q\in [1, \infty]$ such that $\frac{1}{p} + \frac{1}{q} = 1$. 

With the symbol  $\calM(\prX; \C^d)$ ($\calM(\prX)$ in the case $d{=}1$)
we denote the
space of $\C^d$-valued measurable functions on $\prX$, modulo
equality almost everywhere. 
 We shall tacitly identify  $\calM(\prX; \C^d)$ with 
$\calM(\prX)^d$ and use the notation
\[ \bff = (f_1, \dots, f_d)
\]
to denote functions into $\C^d$. For a set $M \subseteq \C^d$ and
$\bff= (f_1, \dots, f_d) \in \calM(\prX;\C^d)$ as above, we write
``$(f_1, \dots, f_d)\in M\quad \text{almost everywhere}$'' 
shorthand for: ``$(f_1(x), \dots, f_d(x))\in M$ for $\mu$-almost all 
$x\in X$.''   By abuse of notation, if $F: \C^d\to \C$ 
is measurable
and $\bff \in \calM(\prX;\C^d)$ 
we write $F(\bff)$ to denote the function $F\nach \bff$, i.e., $F(\bff)(x) = F( f_1(x), \dots, f_d(x))$.

The letters $K, L, \dots $ usually denote compact and sometimes
locally compact  Hausdorff spaces.
We abbreviate this by simply saying that $K,\:L,\dots$ are (locally) compact.
If $K$ is locally compact, then $\Cc(K)$ denote the space
of continuous functions on $K$ with  compact support, and 
$\Co(K)$ is the sup-norm closure  of $\Cc(K)$ within the Banach
space of all bounded continuous functions. If $K$ is compact, then of course
$\Cc(K) = \Co(K) = \Ce(K)$.  

If $K$ is (locally) compact then, by the Riesz representation theorem,
 the dual space of $\Ce(K)$ ($\Co(K)$) is isometrically and lattice isomorphic
to $\eM(K)$, the space of complex regular  Borel measures on $K$, 
with the total variation (norm) as absolute value (norm). 
A {\em (locally) compact measure space} is  a pair $(K,\nu)$ where $K$ is 
(locally) compact and $\nu$ is a positive regular Borel measure on $K$.
(If $K$ is locally compact, the measure $\nu$ need not be finite.)

We work with complex Banach spaces by default. In particular, 
$\Ell{p}$-spaces have to be understood as consisting of complex-valued 
functions. For an operator $T$ with domain and range being
spaces of complex-valued functions, 
the {\em conjugate operator} is defined by $\conj{T}f := \conj{T\conj{f}}$,
and the {\em real part} and {\em imaginary part} are defined by
\[ \Re T := \tfrac{1}{2} (T + \conj{T}) \quad \text{and}\quad
\Im  T := \tfrac{1}{2\ui} (T - \conj{T}), 
\]
respectively. 
For Banach spaces $E$ and $F$ we use the symbol $\BL(E;F)$ to denote
the space of bounded linear operators from $E$ to $F$ and 
$E'= \BL(E;\C)$ for the dual space. The dual of an operator
$T \in \BL(E;F)$ is denoted by $T'\in \BL(F',E')$.

If $K$ is locally compact, $\prX = (X, \Sigma, \mu)$ is a measure space and $T: \Cc(K) \to \Ell{1}(\prX)$ is 
a linear operator, then $T'\mu$ denotes the linear functional on $\Cc(K)$ defined by 
\[  \tdprod{f}{T'\mu} := \int_X Tf \, \ud{\mu}\qquad (f\in \Cc(K)).
\]
If $T$ is bounded for the uniform norm on $\Cc(K)$ then  $T'\mu$ is bounded too, 
and we identify it with the complex regular Borel measure in $\eM(K)$. If $T$ is not bounded but
positive, then, again by the Riesz representation  theorem,
$T'\mu$ can be identified with a  positive (but infinite) regular Borel measure on $K$.

At some places we use some basic notions of Banach lattice
theory (e.g., lattice homomorphism, ideal, order completeness). 
The reader unfamiliar with this terminology can consult 
\cite[Chap.{} 7]{EFHNPre} for a brief account. However,
the only Banach lattice that appears here and is not a function space will be
$\eM(K)$, where $K$ is locally compact.

\section{Main Results}\label{s.mr}

An {\emdf absolute contraction}, 
or a  {\emdf Dunford--Schwartz operator},  over a measure space $\prX$ 
is an operator $T: \Ell{1} \cap \Ell{\infty} \to \Ell{1} + \Ell{\infty}$
satisfying $\norm{Tf}_p \le \norm{f}_p$ for  $p=1$ and $p = \infty$. 
It is then well-known that 
$T$ extends uniquely and consistently to linear contraction operators
$T_p: \Ell{p} \to \Ell{p}$ for $1 \le p < \infty$, and
$T_\infty: \Ell{(\infty)} \to \Ell{(\infty)}$, where
$\Ell{(\infty)}$ is the closed linear hull of $\Ell{1} \cap \Ell{\infty}$
within $\Ell{\infty}$. 
It is common to use the single symbol  $T$ for each
of the operators $T_p$.

An absolute contraction  $T$ is {\emdf sub-Markovian} if
it is {\emdf positive}, i.e., if $Tf\ge 0$ whenever $f\ge 0$, $f\in \Ell{1}
\cap \Ell{\infty}$. (Then also its canonical extension $T_p$ to 
$\Ell{p}$, $1\le p < \infty$  and $\Ell{(\infty)}$, $p = \infty$, is positive.) 
This terminology is coherent with \cite[Def.{} 2.12]{Ouhabaz2005}.

An absolute contraction $T$  is called {\emdf Markovian}, if it satifies
\[ f \le b\car \quad \Dan \quad Tf \le b \car
\]
for every $b\in \R$ and $f\in \Ell{1}\cap \Ell{\infty}$.  (Here, $\car$ is the constant function with value equal to $1$.) 
In particular, $T$ is positive, i.e., sub-Markovian. If the measure space $\prX$ is finite, 
an absolute contraction is Markovian if and only if $T$ is positive and $T\car = \car$. This is easy to see,
cf.{} \cite[Lemma 3.2]{Haase2007}.

An operator $T: \Ell{1} \cap \Ell{\infty} \to \Ell{1} + \Ell{\infty}$ is
{\emdf symmetric} if 
\[ \int_\prX Tf \cdot \conj{g}  = \int_\prX f \cdot  \conj{Tg}\, 
\]
for all $f,g \in \Ell{1} \cap \Ell{\infty}$. A symmetric operator
is an absolute contraction if and only if it is $\Ell{\infty}$-contractive
if and only if it is $\Ell{1}$-contractive; and in this case
the canonical extension to $\Ell{2}$ is a bounded self-adjoint operator.

\medskip


\noindent
A (strongly continuous)
{\emdf absolute contraction semigroup} over $\prX$ 
is a family $(S_t)_{t\ge 0}$ of
absolute contractions on $\prX$ 
such that $S_0 = \Id$, $S_{t+s} = S_t S_s$ for all 
$t, s \ge 0$ and 
\begin{equation}\label{mr.eq:C_0}
 \norm{f - S_tf}_p \to 0 \quad \text{as}\quad  t\searrow 0
\end{equation}
for all $f\in \Ell{1}\cap \Ell{\infty}$ and all $1\le p < \infty$. 
It follows that the operator family $(S_t)_{t\ge 0}$ can be considered
a strongly continuous semigroup on each space $\Ell{p}$, $1\le p < \infty$. 
{\em We shall always assume this continuity property even when it is
not explicitly mentioned.}
An absolute contraction semigroup $(S_t)_{t\ge 0}$ is called 
a  {\emdf symmetric contraction semigroup} 
({\emdf symmetric (sub-)Markovian semigroup}) if
each operator $S_t$, $t \ge 0$, is symmetric (symmetric and (sub-)Markovian).

\begin{rems}
\begin{aufziii}
\item A symmetric sub-Markovian semigroup is called a ``symmetric diffusion semigroup'' in the classical text
\cite{Stein1970}. It appears that the ``diffusion semigroups'' of operator space theory \cite[Def.{} 2]{Kriegler2011}
lack the property of
positivity, and hence do not specialize to Stein's concept in the commutative case, but rather to what
we call ``symmetric contraction semigroups'' here.

\item As Voigt \cite{Voigt1992} has shown, the strong continuity assumption \eqref{mr.eq:C_0} 
for $p \neq 2$ is a consequence of the case $p = 2$ together with 
the requirement that all operators $S_t$ are $\Ell{p}$-contractions.
\end{aufziii}
\end{rems}

\noindent
Given an absolute contraction semigroup $(S_t)_{t\ge 0}$ one can consider, for 
$1\le p < \infty$, the negative generator $-A_p$ of the strongly continuous 
semigroup $(S_t)_{t\ge 0}$ on $\Ell{p}$, defined by 
\begin{align*}
\dom(A_p) =&  \{ f\in \Ell{p} \st \lim_{t \searrow 0} \tfrac{1}{t}(f - S_tf)
\,\,\text{exists in $\Ell{p}$}\},\\
A_p f & = \lim_{t \searrow 0} \tfrac{1}{t}(f - S_tf).
\end{align*}
The operators $A_p$ are compatible for different indices $p$, a fact 
which is easily seen by looking at the resolvent of $A_p$
\[ (\Id + A_p)^{-1}f = \int_0^\infty \ue^{-t} S_tf\, \ud{t} 
\qquad (f\in \Ell{p}, \, 1\le p <\infty).
\] 
Hence, it is reasonable to drop the index $p$ and simply write $A$ instead
of $A_p$.


\medskip
\noindent
In order to formulate the main result, we first look 
at the very special case that the underlying measure
space consists of two atoms with equal mass.
Let this (probability) space be denoted by  $\prZ_2$, i.e.,
\[ \prZ_2 := ( \{0,1\}, 2^{\{0, 1\}}, \zeta_2).
\] 
Then, for $1\le p < \infty$, $\Ell{p}(\prZ_2) = \C^2$ with norm
\[ \bigl\| \begin{pmatrix} z_1 \\ z_2 \end{pmatrix} \bigr\|_p^p = \tfrac{1}{2}
 ( \abs{z_1}^p + 
\abs{z_2}^p).
\]
The scalar product on the Hilbert space $H = \Ell{2}(\prZ_2)$
is 
\[   \begin{pmatrix} z_1 \\ z_2 \end{pmatrix} 
\cdot_{\prZ_2}
\begin{pmatrix} w_1 \\ w_2 \end{pmatrix} 
= \tfrac{1}{2}( z_1\conj{w_1} + z_2\conj{w_2}).
\]
Symmetric operators on $\Ell{2}(\prZ_2)$ are represented by matrices
\[ T = \begin{pmatrix} a & \conj{w}  \\ w & b \end{pmatrix}
\]
with $a, b\in \R$. The property that $T$ is an absolute contraction
is equivalent with the conditions $\abs{a} +\abs{w} \le 1$ and
$\abs{b} + \abs{w} \le 1$. Thus, the absolute contractions on $\prZ_2$
form a closed convex set
\[ C_2 := \Bigl\{ \begin{pmatrix} a & \conj{w}  \\ w & b \end{pmatrix} \,
\big| \, a, b \in \R,\, w\in \C,\, \max\{ \abs{a}, \abs{b}\} \le 1- \abs{w}
\Bigr\},
\]
and it is easy to see that each matrix
\[ E_\lambda:= \begin{pmatrix} 0 & \conj{\lambda}  \\ \lambda &  0\end{pmatrix},
\qquad \lambda\in \torus,
\] 
is an extreme point of $C_2$.
We can now formulate the desired (meta-)theorem. 

\begin{thm}[Symmetric Contraction Semigroups]\label{mr.t:main}
Let $m,\:d\in \N$, $1 \le p < \infty$ and let, for each
$1 \le j \le m$,  $\Eff_j,\: \Gee_j: \C^d \to \C$
be measurable functions. 
For any  generator $-A$ 
of a symmetric contraction semigroup over a 
measure space $\prX$
consider the following statement:\\
``For all measurable functions $\bff \in \calM(\prX;\C^d)$
such that $\Eff_j(\bff)\in \dom(A_p)$ and
$\Gee_j(\bff) \in \Ell{q}(\prX)$ for all $1\le j \le m$:
\[ 
\sum_{j=1}^m \re \int_\prX A \Eff_j(\bff) \cdot \Gee_j(\bff)
\,\,\ge\,\, 0." 
\]
Then this statement holds true provided it holds true whenever
$\prX$ is replaced by $\prZ_2$ and $A$ is replaced
by $\Id - E_\lambda$, $\lambda\in \torus$.
\end{thm}

If, in addition, the semigroup is sub-Markovian, we have an even
better result. In slightly different form (but with with more or less the same method),
this result has been  obtained by Huang in \cite[Theorem 2.2]{Huang2002}.

\begin{thm}[Sub-Markovian Semigroups]\label{mr.t:main-subM}
Let $m,\:d\in \N$, $1 \le p < \infty$ and let, for each
$1 \le j \le m$,  $\Eff_j,\: \Gee_j: \C^d \to \C$
be measurable functions. 
For any  generator $-A$ 
of a symmetric sub-Markovian semigroup over a 
measure space $\prX$
consider the following statement:\\
``For all measurable functions $\bff \in \calM(\prX;\C^d)$
such that $\Eff_j(\bff)\in \dom(A_p)$ and
$\Gee_j(\bff) \in \Ell{q}(\prX)$ for all $1\le j \le m$:
\[ 
\sum_{j=1}^m \re \int_\prX A \Eff_j(\bff) \cdot \Gee_j(\bff)
\,\,\ge\,\, 0." 
\]
Then this statement holds true provided it holds true whenever
$\prX$ is replaced by $\prZ_2$ and $A$ is replaced
by $\Id - E_1$ and by $\Id$.
\end{thm}

The second condition here (that the statement holds
for $\prZ_2$ and $A = \Id$) just means that the scalar inequality
\[ \sum_{j=1}^m \Re F_j(x) G_j(x) \ge 0 
\]
holds for all $x\in \C^d$, cf. Lemma \ref{prf.l:scalarcase} below.

\medskip
\noindent
Finally, we suppose that the measure space $\prX$ is finite
and the semigroup is {\em Markovian}, i.e., $S_t\ge 0$  and $S_t \car = \car$ 
for each $t\ge 0$. Then we have an even simpler criterion.

\begin{thm}[Markovian Semigroups]\label{mr.t:main-M}
Let $m,\:d\in \N$, $1 \le p < \infty$ and let, for each
$1 \le j \le m$,  $\Eff_j,\: \Gee_j: \C^d \to \C$
be measurable functions. 
For any  generator $-A$ 
of a symmetric Markovian semigroup over a  
measure space $\prX$
consider the following statement:\\
``For all measurable functions $\bff \in \calM(\prX;\C^d)$
such that $\Eff_j(\bff)\in \dom(A_p)$ and
$\Gee_j(\bff) \in \Ell{q}(\prX)$ for all $1\le j \le m$:
\[ 
\sum_{j=1}^m \re \int_\prX A \Eff_j(\bff) \cdot \Gee_j(\bff)
\,\,\ge\,\,0." 
\]
Then this statement holds true provided it holds true whenever
$\prX$ is replaced by $\prZ_2$ and $A$ is replaced
by $\Id - E_1$.
\end{thm}

\noindent
The proofs of Theorems \ref{mr.t:main}--\ref{mr.t:main-M} 
are
completed in Section  \ref{s.prf} below after we have performed
some preparatory
reductions (Section \ref{s.red}) and provided some results from abstract
operator theory (Section \ref{s.ot}).

\section{Reduction Steps}\label{s.red}

In this section we shall formulate and prove three results that, 
when combined, reduce the proof of Theorem \ref{mr.t:main} 
to the case when $\prX= (K,\mu)$ is a compact measure space,
$\mu$ has full support,  $\Ell{\infty}(\prX) = \Ce(K)$, and 
$A= \Id - T$, where $T$ is a single symmetric absolute contraction on $\prX$.
These steps are, of course, well-known, but for the convenience 
of the reader we discuss them in some detail.

\subsection{Reduction to Bounded Operators}

Suppose that $(S_t)_{t\ge 0}$ is an absolute contraction semigroup
on $\prX$ with generator $-A$. Then each operator 
$- (\Id- S_\varepsilon)$ is itself the (bounded) generator of 
a (uniformly continuous) 
absolute contraction semigroup $\big(\ue^{-t(\Id - S_\veps)}\big)_{t\ge 0}$
on $\prX$. By definition of $A$, 
\[ \frac{1}{\veps} ( \Id - S_\veps) g \to Ag \quad \text{as}\quad
\veps \searrow 0
\]
in $\Ell{p}$ for $g\in \dom(A_p)$. We thus have the following
first reduction result.

\begin{prop}\label{red.p:bdd}
Let $m,\:d\in \N$, $1 \le p < \infty$ and let, for each
$1 \le j \le m$,  $\Eff_j,\: \Gee_j: \C^d \to \C$
be measurable functions. 
For any  generator $-A$ 
of an absolute contraction semigroup $(S_t)_{t\ge 0}$ over a 
measure space $\prX$
consider the following statement:\\
``For all measurable functions $\bff\in \calM(\prX; \C^d)$
such that $\Eff_j(\bff)\in \dom(A_p)$ and
$\Gee_j(\bff) \in \Ell{q}(\prX)$ for all $1\le j \le m$:
\[ 
\sum_{j=1}^m \re \int_\prX A \Eff_j(\bff) \cdot \Gee_j(\bff)
\,\,\ge\,\, 0." 
\]
Then this statement holds true provided it holds true whenever
$A$ is replaced by $\Id - S_\veps$, $\veps > 0$.
\end{prop}

\noindent
Note that in the case $A =\Id - T$, the condition 
$\Eff_j(\bff)\in \dom(A_p)$ just asserts that 
$\Eff_j(\bff)\in \Ell{p}$.



\subsection{Reduction to a Finite Measure Space}

Now it is shown that one may confine
to finite measure spaces. 
For a given 
measure space
$\prX = (X, \Sigma, \mu)$, the set 
\[ \Sigma_{\fin} := \{ B \in\Sigma \st \mu(B) < \infty \}
\]
is directed with respect to set inclusion. 
For asymptotic statements with respect to this 
directed set we use the abbreviation  ``$B \to X$''.  
The multiplication operators 
\[ M_B: \calM(\prX;\C^d) \to \calM(\prX;\C^d), \qquad M_B \bff := 
\car_B \cdot \bff
\]
form a net, with $M_B \to \Id$ strongly on $\Ell{p}$ as $B \to X$ and 
$1\le p < \infty$. 
It follows that for a given absolute contraction $T$ on $\prX$
and functions $f\in \Ell{p}(\prX)$ and $g\in \Ell{q}(\prX)$
\[ \int_\prX (\Id - T) M_B f \cdot (M_B g) 
\to  \int_\prX (\Id - T)f \cdot g 
\quad \text{as $B \to X$}.
\]
For given $B \in \Sigma_\fin$ we form the
finite measure space $(B, \Sigma_B, \mu_B)$, where
$\Sigma_B := \{ C \in \Sigma \st C \subseteq B\}$ and 
$\mu_B := \mu|_{\Sigma_B}$ . Then we have the 
extension operator
\[ \Ext_B: \calM(B; \C^d) \to \calM(\prX;\C^d),\qquad 
\Ext_B\bff = \begin{cases} \bff & \text{on $B$}\\ 0 & \text{on $X\ohne B$},
\end{cases}
\]
and the restriction operator 
\[ \Res_B: \calM(\prX; \C^d) \to \calM(B; \C^d),\qquad \Res_Bf := f|_B.
\]
Note that $\Ext_B\, \Res_B = M_B$ and $\Res_B \Ext_B = \Id$ and 
\[ \int_B \Res_B f \, \ud{\mu_B} = \int_X M_B f \, \ud{\mu}
\qquad (f\in \Ell{1}(\prX)).
\]
A short computation yields that $\Res_B^* = \Ext_B$ between the
respective $\Ell{2}$-spaces. Hence, if $T$ is a
(symmetric) absolute contraction on $\prX= (X, \Sigma,\mu)$, then 
the operator
\[ 
T_B:=  \Res_B\:T \:\Ext_B
\]
is a (symmetric)  absolute contraction on $(B, \Sigma_B, \mu_B)$.
Another short computation reveals that 
\[ \int_X  (\Id - T)M_Bf \cdot (M_B g)\, \ud{\mu}
= \int_B (\Id_{\Ell{p}(B)} - T_B) (\Res_Bf) \cdot (\Res_B g) \, \ud{\mu_B}
\]
whenever $f\in \Ell{p}(\prX)$ and $g\in \Ell{q}(\prX)$.
Finally, suppose that $\Eff: \C^d \to \C$ is measurable
and suppose that $\bff\in \calM(\prX;\C^d)$ is 
such that  $\Eff(\bff) \in \Ell{p}(\prX)$. Then 
\[ \Res_B[ \Eff(\bff)] = \Eff( \Res_B \bff)
\in \Ell{p}(B).
\]
Combining all these facts yields our second reduction result.

\begin{prop}\label{red.p:fms}
Let $m,\:d\in \N$, $1 \le p < \infty$ and let, for each
$1 \le j \le m$,  $\Eff_j,\: \Gee_j: \C^d \to \C$
be measurable functions.
For any   absolute contraction  $T$ 
over a measure space $\prX= (X, \Sigma, \mu)$
consider the following statement:\\
``For all measurable functions $\bff \in \calM(\prX; \C^d)$
such that $\Eff_j(\bff)\in \Ell{p}(\prX)$ and
$\Gee_j(\bff) \in \Ell{q}(\prX)$ for all $1\le j \le m$:
\[ 
\sum_{j=1}^m \re \int_\prX (\Id - T) \Eff_j(\bff)  \cdot \Gee_j(\bff)
\,\,\ge\,\, 0." 
\]
Then this statement holds true provided it holds true whenever
$\prX= (X, \Sigma, \mu)$ is replaced by 
$(B, \Sigma_B, \mu_B)$ and 
$T$ is replaced by $T_B$, where $B \in \Sigma_\fin$.
\end{prop}

Finally, we observe that if $T$ is sub-Markovian (=positive) or Markovian, then so is each of the operators
$T_B = \Res_B\: T\: \Ext_B$, $B \in \Sigma_\fin$.

\subsection{Reduction to a Compact Measure Space}\label{s.stone}

In the next step we pass from general finite measure spaces
to {\em compact} spaces with a finite positive Borel measure on it. 

Let $\prX= (X, \Sigma, \mu)$ be a {\em finite} measure space.
The space $\Ell{\infty}(\prX)$ is a
commutative, unital  $C^*$-algebra, whence by the Gelfand--Naimark
theorem there is a compact space $K$, the Gelfand space, 
and an isomorphism of unital  $C^*$-algebras 
\[ \Phi:  \Ell{\infty}(\prX) \to \Ce(K).
\]
In particular, $\Phi$ is an isometry. Since the order structure
is determined by the $C^*$-algebra structure (an element
$f$ is $\ge 0$ if and only if there is $g$ such that $f = g \conj{g}$),
$\Phi$ is also an isomorphism of complex Banach lattices.
The following auxiliary result
is, essentially,  a consequence of the Stone-Weierstrass theorem.

\begin{lem}\label{red.l:fc-cont}
In the situation from above, let $M \subseteq \C^d$ be compact and let
$f_1, \dots, f_d \in \Ell{\infty}(\prX)$ be
such that $(f_1, \dots, f_d) \in M$
$\mu$-almost everywhere. Then $(\Phi f_1, \dots, \Phi f_d) \in M$
everywhere on $K$ and 
\begin{equation}\label{red.eq:fc-cont}
 \Phi\bigl(\Eff(f_1, \dots, f_d)\bigr) = \Eff( \Phi f_1, \dots, \Phi f_d)
\end{equation}
for all continuous functions $\Eff \in \Ce(M)$.
\end{lem}

\begin{proof}
Supose first that $M = \Ball[0,r] := \{ x\in \C^d \st \norm{x}_\infty \le r\}$
for some $r > 0$. Then the condition ``$(f_1, \dots, f_d)\in M$ 
almost everywhere''  translates into the inequalities $\abs{f_j} \le r \car$
(almost everywhere) for all $j = 1,\dots, d$, and hence one has also 
$\abs{\Phi f_j} \le r \Phi\car = r \car$ (pointwise everywhere) for all $j = 1,\dots, d$.
It follows that $\Eff(\Phi f_1, \dots, \Phi f_d)$ is well-defined. 

\noindent 
Now, the set of functions $\Eff\in \Ce(M)$ such that \eqref{red.eq:fc-cont}
 holds is a closed conjugation-invariant 
subalgebra of $\Ce(M)$ that separate the points
and contains the constants. Hence, by the
Stone-Weierstrass theorem, it is all of $\Ce(M)$. 

\noindent
For general $M$ one can proceed in the same way provided one
can assure that $(\Phi f_1, \dots, \Phi f_d)\in M$ everywhere on $K$. Let 
$y \in \C^d \ohne M$ and let $\Eff$ be any continuous function
with compact support on $\C^d$ such that $\Eff= 0$ on $M$ and 
$\Eff(y) = 1$. Let $r > 0$ by so large  that $M \subseteq \Ball[0,r]$
and consider $\Eff$ as a function on $\Ball[0,r]$. Then 
$0 = \Phi(0) = \Phi( \Eff(f_1, \dots, f_d)) = 
\Eff( \Phi f_1, \dots, \Phi f_d)$, whence $y$ cannot be in the image of 
$( \Phi f_1, \dots, \Phi f_d)$. 
\end{proof}

\noindent
By the Riesz-Markov representation theorem, there is a unique
regular Borel measure $\nu$ on $K$  such that 
\[ \int_\prX f  = \int_K \Phi f\, \ud{\nu} 
\]
for all $f\in \Ell{\infty}(\prX)$. It follows 
from Lemma \ref{red.l:fc-cont} that $\abs{\Phi f}^p = \Phi( \abs{f}^p)$
for every $1\le p < \infty$ and every 
$f\in \Ell{\infty}(\prX)$. Therefore, $\Phi$
is an isometry with respect to each $p$-norm. It follows
that $\Phi$ extends to an isometric (lattice) isomorphism
\[ \Phi: \Ell{1}(\prX) \to \Ell{1}(K,\nu).
\]
It is shown in Appendix \ref{a.embed} that $\Phi$, furthermore,
extends canonically (and uniquely) 
to a unital $*$-algebra and lattice isomorphism 
\[ \Phi: \calM(\prX) \to \calM(K,\nu).
\]
The compact measure space $(K, \nu)$  (together with the mapping $\Phi$) 
is called the {\emdf Stone model} of the probability space $\prX$. 
Note that  under the lattice isomorphism $\Phi$ 
the respective $\Ell{\infty}$-spaces
must correspond to each other, whence it follows that $\Ell{\infty}(K,\mu) = 
\Ce(K)$ in the obvious sense.

\medskip
\noindent
We use the canonical extension to vector-valued functions
$\Phi: \calM(\prX;\C^d) \to \calM(K,\nu; \C^d)$ of the Stone model.
By Theorem \ref{embed.t:fc-gen}, 
\[ \Phi \bigl( \Eff(\bff)\bigr) = 
\Eff( \Phi \bff)  \qquad \text{$\nu$-almost everywhere}
\]
for all measurable functions $\bff = (f_1, \dots, f_d) \in \calM(\prX; \C^d)$ 
and  all measurable functions $F: \C^d \to \C$. Hence, we 
arrive at the next reduction result.

\begin{prop}\label{red.p:compact}
Let $m,\:d\in \N$, $1 \le p < \infty$ and let, for each
$1 \le j \le m$,  $\Eff_j,\: \Gee_j: \C^d \to \C$
be measurable functions.
For any  absolute contraction  $T$ 
over a probability space $\prX$
consider the following statement:\\
``For all measurable functions $\bff \in \calM(\prX;\C^d)$
such that $\Eff_j(\bff)\in \Ell{p}(\prX)$ and
$\Gee_j(\bff) \in \Ell{q}(\prX)$ for all $1\le j \le m$:
\[ 
\sum_{j=1}^m \re \int_\prX [(\Id - T) \Eff_j(\bff)] \cdot \Gee_j(\bff)
\,\,\ge\,\, 0." 
\]
Then this statement holds true provided it holds true if
$\prX$ is replaced by $(K, \nu)$
and $T$ is replaced by $\Phi T\Phi^{-1}$, where $(K,\nu)$ and 
\[ \Phi: \calM(\prX) \to \calM(K,\nu)
\]
is the Stone model of $\prX$.
\end{prop}

As in the reduction step before, we observe that 
the properties of being symmetric, sub-Markovian or Markovian 
are preserved during  the reduction process, i.e.,
in passing from $T$ to $\Phi^{-1}T\Phi$.


\begin{rem}
In the late 1930's and beginning 1940's, 
several representation results for abstract structures 
were developed first by Stone \cite{Stone1937} (for Boolean algebras),
then by Gelfand \cite{Gelfand1939,GelfandNeumark1943} (for normed algebras) 
and Kakutani  \cite{Kakutani1941a,Kakutani1941b} (for $AM$- and $AL$-spaces).
However, it is hard to determine when for 
the first time there was made effective
use of these results in a context similar to ours.
Halmos in his paper  \cite{Halmos1949} 
on a theorem of Dieudonn\'e on measure disintegration employs the
idea but uses Stone's original theorem. A couple of years later,
Segal \cite[Thm.{} 5.4]{Segal1951} revisits Dieudonn\'e's  theorem 
and gives a proof based on algebra representations. (He does not mention
Gelfand--Naimark, but only says ``by well-known results''.)

In our context, the idea --- now through the Gelfand--Naimark theorem ---
was employed by Nagel and Voigt
\cite{NagelVoigt1996} in order to simplify arguments in the proof
of Liskevich and Perelmuter  \cite{LiskevichPerelmuter1995}
on the optimal angle of analyticity in the sub-Markovian case, see Section
\ref{s.soa} below. Through Ouhabaz' book \cite{Ouhabaz2005} it has become
widely known in the field, and also 
Carbonaro and \Dragicevic\ \cite[p.19]{CarbDrag13Pre} use this idea.
\end{rem}

\section{Operator Theory}\label{s.ot}

In order to proceed with the proof of the main theorem 
(Theorem \ref{mr.t:main}) we need to provide some
results from the theory of operators of the form
$T: \Ce(K) \to \Ell{1}(L,\mu)$, where $K$ and $L$ are compact.\footnote{The case
that $K$ and $L$ are locally compact is touched upon in some additional remarks.}  For the
application to symmetric contraction semigroups as considered in the previous
sections, we only need the case that $\Ce(L) = \Ell{\infty}(L,\mu)$,
and this indeed would render simpler some of the proofs below. 
However, a restriction to this case is artificial, and we develop
the operator theory in reasonable generality.

\subsection{The Linear Modulus}

In this section we introduce the linear modulus of
an order-bounded operator $T: \Ce(K) \to \Ell{1}(\prX)$.
This can be treated in the framework of general
Banach lattices, see \cite[Chapter IV,\S1]{Schaefer1974},
but due to our concrete situation, 
things are a little easier than in an  abstract setting.


\medskip
\noindent
Let $\prX = (X, \Sigma, \mu)$ be a  measure space and let
$K$ be compact. A linear  operator $T: \Ce(K) \to \Ell{1}(\prX)$
is called {\emdf order-bounded} if for each $0\le f\in \Ce(K)$ there
is $0 \le h\in  \Ell{1}(\prX)$ such that 
\[ \abs{Tu}\le h \quad \text{for all $u \in \Ce(K)$ with $\abs{u}\le f$.}
\]
And $T$ is called {\emdf regular} if it is a linear combination 
of positive operators. It is clear that each regular operator
is order-bounded. The converse also holds, by the following construction.

Suppose that $T: \Ce(K) \to \Ell{1}(\prX)$ is order-bounded.
Then, for  $0 \le f\in \Ce(K)$ let
\begin{equation}\label{ot.eq.mod-def}
 \linabs{T}f :=\sup \{ \abs{Tg} \st g\in \Ce(K), \abs{g}\le f\}
\end{equation}
as a supremum in the lattice sense. (This supremum exists since
the set on the right hand side is order bounded by hypothesis and
$\Ell{1}$ is order complete, see \cite[Cor.{} 7.8]{EFHNPre}.)

\begin{lem}\label{ot.l:mod-welldef} 
Suppose that $T: \Ce(K) \to \Ell{1}(\prX)$ is order-bounded. Then
the mapping 
$\abs{T}$ defined by \eqref{ot.eq.mod-def} 
extends uniquely to a positive operator
\[ \abs{T}: \Ce(K) \to \Ell{1}(\prX).
\] 
Moreover, the following assertions hold:
\begin{aufzi}
\item $\abs{Tf}\le \abs{T}\abs{f}$ for all $f\in \Ce(K)$. 
\item $\norm{T} \le \norm{\abs{T}}$,
\item $\conj{T}$ is order-bounded and $\abs{\conj{T}} = \abs{T}$.
\item If $S: \Ce(K) \to \Ell{1}(\prX)$ is order-bounded, then 
  $S + T$ is also order-bounded, and $\abs{S+T} \le \abs{S}+ \abs{T}$. 
\end{aufzi}
\end{lem}

\noindent
The operator $\abs{T}: \Ce(K)\to \Ell{1}(\prX)$ whose existence is asserted in the theorem is called
the {\emdf linear modulus} of $T$. 

\begin{proof}
For the first assertion, it suffices to show that $\abs{T}$ is additive and
positively homogeneous. The latter is straightforward, so consider
additivity. Fix $0 \le f,g\in \Ce(K)$ and  let $u \in \Ce(K)$ with
$\abs{u} \le f + g$. Define 
\[ u_1 = \frac{fu}{f+g} ,\quad
u_2 = \frac{gu}{f+g}, 
\]
where $u_1= u_2 =0$ on the set $\set{f +g = 0}$. Then $u_1,\: u_2\in \Ce(K)$,
$\abs{u_1} \le f$, $\abs{u_1} \le g$ and $u_1 + u_2 = u$.
Hence
\[ \abs{Tu} \le \abs{Tu_1} + \abs{Tu_2}\le \linabs{T}f + \linabs{T}g
\]
and taking the supremum with respect to $u$ we obtain $\linabs{T}(f+g) \le 
\linabs{T}f + \linabs{T}g$.
Conversely, let $u,\:v\in\Ce(K)$ with $\abs{u} \le f$ and $\abs{v} \le g$. 
Then, for any $\alpha \in \C^2$ with  $\abs{\alpha_1} + \abs{\alpha_2} \le 1$
we have  $\abs{\alpha_1 u + \alpha_2v} \le f +g$, and hence
\begin{align*}
\abs{Tu} + \abs{Tv}
= \sup_\alpha \abs{ \alpha_1 Tu + \alpha_2 Tv}
=  \sup_\alpha \abs{ T( \alpha_1 u + \alpha_2 v)} \le \linabs{T}(f+g).
\end{align*}
Taking suprema with respect to $u$ and $v$ we arrive at 
$\linabs{T}f + \linabs{T}g \le \linabs{T}(f+g)$. 
The remaining statements are now easy to establish. 
\end{proof}

\noindent
Suppose that 
$T: \Ce(K) \to \Ell{1}(\prX)$ is order-bounded, so that 
$\abs{T}$ exists. Then, by Lemma \ref{ot.l:mod-welldef}, 
also $\Re T$ and $\Im T$ are
order-bounded. If $T$ is real, i.e., if $T = \conj{T}$, then 
clearly $T \le \abs{T}$, and hence $T = \abs{T} - (\abs{T} - T)$ is regular.
It follows that every order-bounded operator is regular. (See also 
\cite[IV.1, Prop.s 1.2. and 1.6]{Schaefer1974}.) 

\smallskip
\noindent
Let us turn to another characterization of order-boundedness. 
If $T: \Ce(K) \to \Ell{1}(\prX)$ is order-bounded and $\abs{T}$ is its
linear modulus, we denote by $\abs{T}'\!\mu$ the unique element $\nu \in \eM(K)$ 
such that 
\[ \int_K f\, \ud{\nu} = \int_\prX \linabs{T}f\quad \text{for all $f\in \Ce(K)$}.
\]
It is then easy to see that $T$ extends to a contraction
$T: \Ell{1}(K,\nu) \to \Ell{1}(\prX)$. We shall see that the existence
of a positive regular Borel measure $\nu$ on $K$ with this property
characterizes the order-boundedness.
The key is the following general result, which has (probably) 
been established first by Grothendieck \cite[p.67, Corollaire]{GrothendieckPT}. 

\begin{lem}\label{ot.l:mod-L1-ineq}
Let $\prX,\: \prY$ be measure spaces and let 
$T: \Ell{1}(\prY) \to \Ell{1}(\prX)$ be a bounded operator. Then for any finite
sequence $f_1, \dots f_n, \in \Ell{1}(\prY)$
\[   \int_\prX  \sup_{1\le j \le n} \abs{Tf_j}  
\le \norm{T}  \int_\prY \sup_{1\le j \le n} \abs{f_j}.
\]
\end{lem}

 \begin{proof}
By approximation, we may suppose that all the functions 
$f_j$ are integrable step functions with respect to one 
finite partition $(A_k)_k$.
We use the variational form
\[    \sup_{1\le j\le n}  \abs{z_j} = \sup \Big\{\abs{ {\sum}_j^n   \alpha_j z_j }
\st  \alpha \in \ell^1_n, \norm{\alpha}_1 \le 1\Big\} 
\]
for complex numbers $z_1, \dots, z_n $. 
Then, with $f_j = \sum_k c_{jk} \car_{A_k}$,
\begin{align*}
\sup_{1\le j\le n}  \abs{Tf_j} & = \sup_\alpha \abs { {\sum}_j^n{\sum}_k \alpha_j c_{jk}T\car_{A_k} }
\\ & \le 
\sup_\alpha   {\sum}_k  \norm{\alpha}_ 1 \big(\sup_{1\le j \le n} \abs{c_{jk}}\big) \abs{T\car_{A_k}}
\\ & =  
{\sum}_k   \big(\sup_{1\le j \le n} \abs{c_{jk}}\big) \abs{T\car_{A_k}}.
\end{align*}
Integrating yields
\begin{align*}
\int_\prX & \sup_{1\le j\le n}  \abs{Tf_j} 
 \le {\sum}_k   \big(\sup_{1\le j \le n} \abs{c_{jk}}\big) \norm{T\car_{A_k}}_1
\\ & \le \norm{T}
{\sum}_k   \big(\sup_{1\le j \le n} \abs{c_{jk}}\big) \norm{\car_{A_k}}_1
\\ & = \norm{T}
\int_\prY {\sum}_k   \big(\sup_{1\le j \le n} \abs{c_{jk}}\big)\car_{A_k}
= \norm{T} \int_\prY \sup_{1\le j\le n}  \abs{f_j}.\qedhere
\end{align*}
\end{proof}

\noindent
We can now formulate the main result of this section.

\begin{thm}\label{ot.t:ob-char}
Let  $\prX= (X, \Sigma, \mu)$ be any measure space and 
$T: \Ce(K) \to \Ell{1}(\prX)$ a linear operator.
Then the following assertions are equivalent:
\begin{aufzii}
\item $T$ is order-bounded.
\item $T$ is regular.
\item There is a positive regular Borel measure $\nu \in \eM(K)$ such that 
$T$ extends to a contraction $\Ell{1}(K,\nu) \to \Ell{1}(\prX)$.
\end{aufzii}
If {\upshape (i)--(iii)} hold, then 
\[  \abs{T}'\!\mu = \min\bigl\{\nu \in \eM_+(K) \st \norm{Tf}_{\Ell{1}(\prX)}
\le \norm{f}_{\Ell{1}(K,\nu)} \,\,\text{for all $f\in \Ce(K)$}\bigr\}.
\]
In particular, if $0 \le \nu \in\eM(K)$ is such that 
 $T$ extends to a contraction $\Ell{1}(K,\nu) \to \Ell{1}(\prX)$, then
so does $\abs{T}$. 
\end{thm}

\begin{proof}
The implications (i)$\Leftrightarrow$(ii)$\Rightarrow$(iii) have already been
established. Moreover, if (i) holds then
it follows from the inequality $\abs{Tf}\le \linabs{T}\abs{f}$ that 
$\norm{Tf}_1\le \norm{f}_{\Ell{1}(K, \nu)}$ with $\nu = \abs{T}'\! \mu$. 

\noindent
On the other hand, suppose (iii) holds and 
that $0\le \nu \in \eM(K)$ is such that 
$\int_\prX \abs{Tf} \le \int_K \abs{f}\, \ud{\nu}$ for all $f\in \Ce(K)$. 
Let $0 \le f\in \Ce(K)$, $n \in \N$  and $u_j \in \Ce(K)$ with $\abs{u_j} \le f$
($1 \le j \le n$). Then, by Lemma \ref{ot.l:mod-L1-ineq},
\[ \int_\prX \sup_{1\le j \le n} \abs{Tu_j} \le \int_K \sup_{1\le j \le n} \abs{u_j}\, 
\ud{\nu} \le \int_K f\, \ud{\nu}.
\]
Now, any upwards directed and norm bounded net in  $\Ell{1}_+$ is 
order-bounded and converges in $\Ell{1}$-norm towards its supremum,
see \cite[Thm.{} 7.6]{EFHNPre}. It follows that $T$ is order-bounded, and
\[ \int_\prX \linabs{T}f \le \int_K f\, \ud{\nu}.
\]
Consequently, $\abs{T}'\!\mu \le \nu$, as claimed. 
\end{proof}

\begin{rems}\label{ot.r:regops}
\begin{aufziii}
\item Suppose that (i)--(iii) of Theorem \ref{ot.t:ob-char} hold.
Then $\abs{T'\mu} \le \abs{T}'\!\mu$, and  
equality holds if and only if $T$ extends to
a contraction $T: \Ell{1}(K,\abs{T'\mu}) \to \Ell{1}(\prX)$.

\item The modulus mapping $T \mapsto \abs{T}$ turns $\BL^r(\Ce(K), \Ell{1}(\prX))$, the set of regular operators,  into a complex Banach lattice
with the norm $\norm{T}_r := \norm{ \abs{T} }$,
 see \cite[Chap.{} IV, \S1]{Schaefer1974}.

\item  All the results of this section
hold {\em mutatis mutandis} for
linear operators $T: \Cc(Y) \to \Ell{1}(\prX)$, where
$Y$ is a locally compact space and $\Cc(Y)$ is the space
of continuous functions on $Y$ with compact support.
\end{aufziii}
\end{rems}

\noindent
The modulus of a linear operator appears already in the seminal
work of Kantorovich \cite{Kantorovich1940}
on operators on linear ordered spaces. 
For operators on an $\Ell{1}$-space the linear modulus was 
(re-)introduced in \cite{ChaconKrengel1964} by Chacon and Krengel  
who probably were  not aware of Kantorovich's work.  Later on, their
construction was generalized to order-bounded operators between 
general Banach lattices by Luxemburg and Zaanen in  \cite{LuxemburgZaanen1971}
and then incorporated by Schaefer in his monograph \cite{Schaefer1974}.

The equivalence of order-bounded and regular operators is of course
a standard lemma from Banach lattice theory.
Lemma \ref{ot.l:mod-L1-ineq}  is essentially equivalent to saying that 
every bounded operator between $\Ell{1}$-spaces is order-bounded.
This has been realized by Gro\-then\-dieck in
\cite[p.66, Prop.{} 10]{GrothendieckPT}. (Our proof
differs considerably from the original one.) 
The equivalence of (i)--(iii)
in Theorem \ref{ot.t:ob-char} can also be 
derived from combining Theorem IV.1.5 
and Corollary 1 of Theorem II.8.9 of \cite{Schaefer1974}. However,
the remaining part of Theorem \ref{ot.t:ob-char} might be new.

\subsection{Integral Representation of Bilinear Forms}

In this section we aim for yet another characterization of
order-bounded operators $T: \Ce(K) \to \Ell{1}(\prX)$ in the case 
that $\prX = (L, \mu)$ is a compact measure space. We shall
see that an operator $T$ is order-bounded if, and only if,
there is a (necessarily unique) complex regular Borel measure
$\mu_T$ on $K \times L$ such that 
\begin{equation}\label{ot.eq:bil-rep}
\int_{K \times L} f\tensor g \, \ud{\mu_T} = \int_L (Tf) \cdot g\, \ud{\mu}
\quad\text{for all $f\in \Ce(K)$ and $g\in \Ce(L)$.}
\end{equation}
This result goes essentially back to Grothendieck's characterization of
``integral'' operators in  \cite[p.141, Thm.{} 11]{GrothendieckPT}, 
but we give ad hoc proofs avoiding the 
tensor product theory. The following simple lemma is the key result here.

\begin{lem}\label{ot.l:bil-rep}
Let $K,L$ be compact spaces. Then, for any bounded operator
$T: \Ce(K) \to \Ce(L)$ and any $\mu \in \eM(L)$ 
there is a unique complex regular Borel measure
$\mu_T \in \eM(K\times L)$ such that \eqref{ot.eq:bil-rep} holds. 
Moreover, $\mu_T\ge 0$ whenever $\mu\ge 0$ and $T \ge 0$. 
\end{lem}

\begin{proof}
The uniqueness is clear since $\Ce(K) \tensor \Ce(L)$ is dense
in $\Ce(K\times L)$. For the existence, let $S: \Ce(K \times L) \to \Ce(L)$
be given by composition of all of the operators in the following chain:
\[ \Ce(K\times L) \cong \Ce(L; \Ce(K)) \stackrel{T^\tensor}{\longrightarrow} 
\Ce(L; \Ce(L)) \cong \Ce(L\times L) \stackrel{D}{\longrightarrow} \Ce(L).
\]
Here, $T^\tensor$ denotes the operator $G \mapsto T\nach G$ and 
$D$ denotes the ``diagonal contraction'', defined by $DG(x) := G(x,x)$ for 
$x\in L$ and $G\in \Ce(L\times L)$.  Then $\mu_T := S'\mu$ satisfies
the requirements, as a short argument reveals.
%
\end{proof}

\begin{rems} \label{ot.rs:bil-rep} 
\begin{aufziii}
\item
The formula \eqref{ot.eq:bil-rep} stays true for all choices of
$f\in \Ce(K)$ and $g$ a bounded measurable function on $L$. 

\item
Our proof of Lemma \ref{ot.l:bil-rep} yields a formula for the
integration of a general $F\in \Ce(K\times L)$ with respect to $\mu_T$:
\[ \int_{K \times L} F(x,y) \,\ud{\mu_T(x,y)}
= \int_{L} \bigl( TF(\cdot, y) \bigr)(y)\, \ud{\mu(y)}.
\]
This means: fix $y \in L$, apply $T$ to the function $F(\cdot,y)$ and 
evaluate this at $y$; then integrate this funtion in $y$ with respect to $\mu$.

\item
Compare this proof of Lemma \ref{ot.l:bil-rep}  with  
the one  given in \cite[p.90/91]{Ouhabaz2005}. 

\item Lemma \ref{ot.l:bil-rep}
remains valid if $K$ and $L$ are merely locally compact, and
$\Ce(\:\cdot\:)$ is replaced by $\Co(\:\cdot\:)$ at each occurrence. 
\end{aufziii}
\end{rems}

\noindent
Combining Lemma \ref{ot.l:bil-rep} with a Stone model leads to the desired
general theorem.

\begin{thm}\label{ot.t:reg=int} 
Let $K$ be compact, $(L,\mu)$ a compact measure space,
and $T: \Ce(K) \to \Ell{1}(L,\mu)$ a linear operator. Then the
following assertions are equivalent:
\begin{aufzii}
\item $T$ is order-bounded.
\item $T$ is regular.
\item $T$ extends to a contraction $\Ell{1}(K,\nu) \to \Ell{1}(L,\mu)$
for some $0 \le \nu \in \eM(K)$.
\item There is a complex regular Borel measure 
$\mu_T \in \eM(K\times L)$ such that \eqref{ot.eq:bil-rep} holds.
\end{aufzii}
In this case, $\mu_T$ from {\upshape (iv)} is unique, and if $\nu$
is as in {\upshape (iii)}, then  $\abs{T}'\!\mu \le \nu$.  
\end{thm}

\begin{proof}
It was shown in Theorem \ref{ot.t:ob-char}
that (i)--(iii) are pairwise equivalent. 

Denote by $\pi_K: K \times L \to K$ the canonical projection. Suppose that (iv) holds and 
let $\nu = (\pi_K)_*\abs{\mu_T}$, i.e., 
\[ \int_K f \, \ud{\nu} = \int_{K \times L} f\tensor \car\, \ud{\abs{\mu_T}}
\qquad (f\in \Ce(K)).
\]
Then, for $f\in \Ce(K)$ and $g\in \Ce(L)$ with $\abs{g}\le 1$, 
\begin{align*}
 \Big| & \int_L Tf\cdot g \, \ud{\mu} \Big|
\le \int_{K\times L}  \abs{f}\tensor \abs{g} \, \ud{\abs{\mu_T}}
 \le \int_{K\times L}  \abs{f}\tensor \car \, \ud{\abs{\mu_T}}
= \int_K \abs{f} \, \ud{\nu}.
\end{align*}
This implies that $T$ extends to a contraction $\Ell{1}(K,\nu) \to 
\Ell{1}(L,\mu)$, whence we have (iii).

\smallskip
\noindent
Now suppose that (i)--(iii) hold. In order to prove (iv) 
define the operator $S: \Ce(K) \to \Ell{\infty}(L,\mu)$ by
\[ Sf  := \begin{cases} \frac{Tf}{\linabs{T}\car} & \text{on\,  
$\set{\linabs{T}\car > 0}$},\\
0  & \text{on\,  $\set{\linabs{T}\car = 0}$}.
\end{cases}
\]
Let  $\Phi: \Ell{1}(L,\mu) \to \Ell{1}(\Omega,\tilde{\mu})$ 
be the Stone model of $(L,\mu)$ (see Section \ref{s.stone} above), and 
let us identify $\Ell{\infty}(L,\mu)$ with $\Ce(\Omega)$ via $\Phi$. 
Then $S: \Ce(K) \to \Ce(\Omega)$ is a positive 
operator. Hence we can apply Lemma \ref{ot.l:bil-rep} to $S$ and
the positive measure $(\linabs{T}\car)\tilde{\mu}$ 
to obtain a positive measure $\rho$ on 
$K \times \Omega$ such that 
\begin{align*}
\int_{K \times \Omega}  f \tensor g\, \ud{\rho}
&= \int_\Omega Sf \cdot g \, \ud{ (\linabs{T}\car)\tilde{\mu}}
= \int_\Omega Sf \cdot \linabs{T}\car \cdot g \, \ud{\tilde{\mu}}
\\ & = \int_\Omega Tf \cdot g \, \ud{\tilde{\mu}} = \int_{L} Tf\cdot g\, \ud{\mu}.
\end{align*}
Finally, let $\mu_T$ be the pull-back of $\rho$ to $K\times L$
via the canonical inclusion map $\Ce(L) \to \Ell{\infty}(L, \mu)=\Ce(\Omega)$.
\end{proof}

\begin{rem}
With a little  more effort one can extend Theorem \ref{ot.t:reg=int} 
to the case of {\em locally compact} (and not necessarily finite) measure spaces $(K,\nu)$
and $(L,\mu)$ instead of compact
ones, cf.{} Remarks \ref{ot.r:regops} and \ref{ot.rs:bil-rep}  above. 
Then the decisive implication 
(ii)$\Rightarrow$(iv) is proved by passing first to open and relatively
compact subsets $U\subseteq K$ and $V\subseteq L$ and considering the
operator $T_{U,V}: \Co(U) \to \Ell{1}(V, \mu)$. By modifying our proof,
one then obtains a measure $\mu_T^{U,V}$ on $U \times V$, and 
finally $\mu_T$ as an inductive limit. (Of course, one has
to speak of Radon measures here.) Compare this to the ad hoc
approach in \cite[Lemma 1.4.1]{Fukushimaetal2011}.

Theorem \ref{ot.t:reg=int}  can also be generalized to the case
that $K$ and $L$ are {\em Polish} (but not necessarily locally
compact) spaces and $\mu$ is a finite positive Borel measure on $L$.
In this case the decisive implication (ii)$\Rightarrow$(iv) is proved
as follows: first, one chooses compact metric models $(K',\nu')$ and $(L',\mu')$
for the
finite Polish measure spaces $(K,\nu)$ and $(L,\mu)$, respectively, 
see \cite[Sec.{} 12.3]{EFHNPre}; by a theorem of von Neumann 
\cite[App.{} F.3]{EFHNPre}, 
the isomorphisms between the original
measure spaces and their models are induced by measurable maps $\vphi:
K'\to K$ and $\psi: L'\to L$, say.   Theorem \ref{ot.t:reg=int}
yields --- for the transferred operator --- 
a representing measure on $K'\times L'$, and
this is mapped by $\vphi \times \psi$ to a representing measure 
on $K \times L$ for the original operator.   
\end{rem}

 \noindent
We now combine the integral Theorem \ref{ot.t:reg=int} 
with the construction of the modulus. We employ the notation 
$\pi_L: K \times L \to L$ for the canonical projection, and identify
\[ \Ell{1}(L,\mu) = \{ \lambda \in \eM(L) \st  \abs{\lambda} \ll \mu\}
\]
with a closed ideal in $\eM(L)$ via the Radon-Nikod\'ym theorem.

\begin{thm}\label{ot.t:rep-mod}
Suppose that $K$ and $L$ are compact spaces and $0 \le \mu \in \eM(L)$.
Then, for any order-bounded operator $T: \Ce(K) \to \Ell{1}(L,\mu)$,
\[ \abs{\mu_T} = \mu_{\abs{T}}.
\]
The  mapping
\[ \BL^r(\Ce(K), \Ell{1}(L,\mu)) \to \eM(K\times L), \qquad T \mapsto \mu_T
\]
is an isometric lattice homomorphism onto the closed ideal 
\[ \{ \rho \in \eM(K\times L) \st \pi_{L\ast}\!\abs{\rho} \in  \Ell{1}(L,\mu) \}
\]
of $\eM(K\times L)$. 
\end{thm}

\begin{proof}
It is clear that the mapping $T\mapsto \mu_T$ is linear, injective
and positive. Hence $\abs{\mu_T}\le \mu_{\abs{T}}$, and therefore 
$\pi_{L*}\abs{\mu_T} \le \pi_{L*} \mu_{\abs{T}} =  (\linabs{T}\car) \mu \in \Ell{1}(L,\mu)$.
Conversely, suppose that 
$\rho \in \eM(K\times L)$ such that $\pi_{L*}\abs{\rho} \in \Ell{1}(L,\mu)$. 
For $f\in \Ce(K)$ consider the linear mapping
\[ T: \Ce(K) \to \eM(L),\qquad (Tf)g := \int_{K \times L} f\tensor g\, \ud{\rho}.
\]
Then $\abs{Tf}\le \norm{f}_{\infty} \pi_{L*}\!\abs{\rho}$, 
whence $Tf \in \Ell{1}(L,\mu)$. 
Hence, by construction,
\[ \int_{K \times L} f\tensor g\, \ud{\rho} = \int_L Tf \cdot g \, \ud{\mu}
\]
for $f\in \Ce(K)$ and $g\in \Ce(L)$. By Theorem \ref{ot.t:reg=int}, 
$T$ is regular. If $\rho$ is positive, then $T$ is positive, too.

\smallskip
\noindent
The proof of the converse inequality
$\mu_{\abs{T}}\le \abs{\mu_T}$ would now follow immediately if we
used the fact (from Remark \ref{ot.r:regops}) 
that the modulus map turns $\BL^r$, the set of regular operators, into 
a complex vector lattice. However, we want to give a different proof here.

\smallskip
\noindent
By a standard argument, it  suffices to establish the inequality
\[\int_L \linabs{T}\car\, \ud{\mu} \le
\int_{K \times L} \car \tensor \car\, \ud{\abs{\mu_T}}.
\]
To this end, define the positive measure $\nu$ on $K$ by 
\[ \int_K f\, \ud{\nu} := \int_{K\times L} f\tensor \car \, \ud{\abs{\mu_T}}
\quad(f\in \Ce(K)).
\]
Given $f\in \Ce(K)$ there is a bounded measurable function
$h$ on $L$ such that $\abs{Tf} = (Tf) h$ and $\abs{h}\le 1$.  Hence,
\[ \int_L \abs{Tf}\, \ud{\mu} 
= \int_L Tf\cdot h\, \ud{\mu} = \int_{K \times L} f \tensor h\, \ud{\mu_T}
\le 
\int_{K \times L} \abs{f} \tensor \car\, \ud{\mu_T} = 
\int_K \abs{f} \, \ud{\nu}.
\]
This means that $T$ extends to a contraction $\Ell{1}(K,\nu)\to \Ell{1}(L,\mu)$.
By Theorem \ref{ot.t:ob-char}, it follows that $\abs{T}'\!\mu \le \nu$, whence
in particular
\[  \int_L \linabs{T}\car\, \ud{\mu} = \int_K \car \ud{(\abs{T}'\!\mu)}
\le \int_K \car \,\ud{\nu} = \int_{K \times L} \car\tensor \car \,
\ud{\abs{\mu_T}}.
\]
This concludes the proof.
\end{proof}

\begin{rem}
One can avoid the use of the bounded measurable  
function $h$ in the second part
of the proof of Theorem \ref{ot.t:rep-mod} by passing to the Stone model
of $\Ell{1}(L,\mu)$. 
\end{rem}

\noindent
In case that $T$ has additional properties,
one can extend the defining formula for the measure $\mu_T$
to some non-continuous functions.

\begin{thm}\label{ot.t:mu_T-Lp}
Let $(K,\nu)$ and $(L,\mu)$ be compact measure spaces, 
and let $T: \Ce(K) \to \Ell{\infty}(L,\mu)$ be a bounded operator
that extends to a bounded operator $\Ell{1}(K,\nu) \to \Ell{1}(L,\mu)$.
Then the formula
\begin{equation}\label{ot.eq:mu_T-Lp}
 \int_L Tf \cdot g\, \ud{\mu} = \int_{K\times L} f\tensor g\, \ud{\mu_T}
\end{equation}
holds for all $f\in \Ell{p}(K, \nu)$, $g\in \Ell{q}(L,\mu)$
and  $1\le p,\:q \le \infty$ with $\frac{1}{p} + \frac{1}{q}=1$.
\end{thm}

\begin{proof}
We may suppose that $T: \Ell{1}(K,\nu)\to \Ell{1}(L,\mu)$ (and hence also $\abs{T}$) is a contraction. 
In a first step, we
claim that the formula \ref{ot.eq:mu_T-Lp} holds for all bounded Baire measurable functions $f,\:g$ on $K,\: L$,
respectively. Indeed, this follows from a standard argument by virtue of the dominated convergence theorem and the fact that
the bounded Baire-measurable functions on a compact space form the smallest set of functions
that contains the continuous ones and is closed under pointwise convergence of uniformly bounded sequences,
see \cite[Thm.{} E.1]{EFHNPre}. 

Replacing $T$ by $\abs{T}$ in \ref{ot.eq:mu_T-Lp} we then can estimate for bounded Baire measurable
functions $f$ and $g$ and $1 < p < \infty$ 
\begin{align*}
\int_{K \times L} & \abs{f\tensor g} \, \ud{\mu_{\abs{T}}} 
= 
\int_{K \times L} (\abs{f}\tensor \car) \cdot (\car \tensor \abs{g})  \, \ud{\mu_{\abs{T}}} 
\\ & \le \Bigl( \int_{K\times L} \abs{f}^p \tensor \car\,\ud{\mu_{\abs{T}}}\Bigr)^\frac{1}{p} 
\cdot \Bigl( \int_{K\times L} \car\tensor \abs{g}^q\,\ud{\mu_{\abs{T}}}\Bigr)^\frac{1}{q}
\\ & =
 \Bigl( \int_{L} \linabs{T}\abs{f}^p\,\ud{\mu}\Bigr)^\frac{1}{p} 
\cdot \Bigl( \int_{L} (\linabs{T}\car)\cdot \abs{g}^q\,\ud{\mu}\Bigr)^\frac{1}{q}
\\ & \le
 \Bigl( \int_{K} \abs{f}^p\,\ud{\nu}\Bigr)^\frac{1}{p} 
\cdot \Bigl( \int_{L} (\linabs{T}\car)\cdot \abs{g}^q\,\ud{\mu}\Bigr)^\frac{1}{q}
= \norm{f}_{\Ell{p}(\nu)} \,\norm{(\linabs{T}\car)^\frac{1}{q} g}_{\Ell{q}(\mu)}.
\end{align*}
It follows that if $A$ is a $\nu$-null Baire set of $K$  and $B$ is a $\mu$-null Baire set of $L$, then
the sets $A \times L$ and $K\times B$ are  $\mu_{\abs{T}}$-null Baire sets of $K \times L$. 
Moreover,  the bilinear mapping $(f,g) \mapsto f \tensor g$
extends to a bounded bilinear mapping 
\[ \Ell{p}(K, \nu) \times \Ell{q}(L,\mu) \to \Ell{1}(K\times L, \mu_{\abs{T}}).
\]
By interpolation, $T$ is $\Ell{p}$-bounded, and hence the bilinear mapping
$(f,g) \mapsto Tf \cdot g$ is a bounded bilinear mapping
$\Ell{p}(K,\nu) \times \Ell{q}(L,\mu) \to \Ell{1}(L,\mu)$. Now
\eqref{ot.eq:mu_T-Lp} 
holds for bounded Baire-measurable functions $f$ and $g$, 
whence by approximation for all $f\in \Ell{p}(K,\nu)$ and $g\in \Ell{q}(L,\mu)$.
(Choose sequences that approximate in norm and almost everywhere.)

\noindent
Finally, consider $p =1$ (the case $q=1$ being similar). 
If $g \in \Ell{\infty}(L,\mu)$ then, by choosing a Baire-measurable representative for $g$ such that $\norm{g}_\infty = 
\norm{g}_{\Ell{\infty}(L,\nu)}$ and using the results from above, we can estimate 
for each $f\in \Ell{\infty}(K, \nu)$,
\begin{align*}
 \int_{K\times L} & \abs{f\tensor g} \, \ud{\mu_{\abs{T}}} = 
\int_{K \times L} (\abs{f}\tensor \car) \cdot (\car \tensor \abs{g})  \, \ud{\mu_{\abs{T}}} 
\\ & 
\le \int_{K \times L} \abs{f}\tensor \car \cdot \norm{g}_\infty  \, \ud{\mu_{\abs{T}}} 
= \norm{ \abs{T} \abs{f} }_{\Ell{1}(L,\nu)} \norm{g}_\infty
\\ & \le \norm{f}_{\Ell{1}(K,\nu)} \, \norm{g}_{\Ell{\infty}(L,\mu)}.
\end{align*}
The assertion then follows by approximation (almost everywhere and in norm) as before. 
\end{proof}


\begin{rem}
If an operator $T: \Ce(K) \to \Ell{1}(L,\mu)$ factors through 
$\Ell{\infty}(L,\mu)$, it is of course order-bounded, and hence
its modulus exists. If, in addition, it factors even through $\Ce(K)$, 
then the existence of $\mu_T$  follows from Lemma \ref{ot.l:bil-rep}
directly and one does not have to pass through the Stone model. 
If $(L,\mu)$ {\em is} already its own Stone model (as is the case 
in the proof of Theorem \ref{mr.t:main} after the reduction step in 
Section \ref{s.stone}) then also $\abs{T}$ factors
through $\Ce(L)$, and hence Lemma \ref{ot.l:bil-rep} is completely sufficient
to construct the measures $\mu_T$ and $\mu_{\abs{T}}$.  
\end{rem}

\noindent
Using modern tensor product terminology, we have
\[ \Ce(K\times L) = \Ce(K) \tensor_\veps \Ce(L) 
\subseteq \Ce(K) \tensor_\veps \Ell{\infty}(L,\mu) 
= \Ce(K) \tensor_\veps \Ell{1}(L,\mu)'.
\]
This implies (via the Stone model of $(L,\mu)$) 
that an operator $T: \Ce(K) \to \Ell{1}(L,\mu)$ 
is ``integral'' (in the sense of Grothendieck) if and only
if there is $\mu_T \in \eM(K\times L)$ such that \eqref{ot.eq:bil-rep} holds. 
Hence, the decisive equivalence of (ii) and (iv) in Theorem \ref{ot.t:rep-mod} 
is essentially \cite[p.141, Thm.{} 11]{GrothendieckPT}. Schaefer incorporates
these results in his systematic study of operators between Banach lattices,
see \cite[IV, Theorem 5.6]{Schaefer1974}.  
However, the property $\abs{\mu_T} = \mu_{\abs{T}}$, essential for 
our application below,  does not appear there. It has been stated
and proved explicitly in \cite[Lemma 30]{CarbDrag13Pre}, but our proof
is different.     


\subsection{The Disintegration Theorem}

In this section we develop further the results of the previous section.
The endpoint will be a ``disintegration'' theorem for
operators of the form $\Id - T$, where 
$T$ is a symmetric absolute contraction over a compact measure space. 

We start with some auxiliary results.

\begin{prop}\label{ot.p:rep-TS}
Let $(K,\nu)$ and $(L,\mu)$ be compact measure spaces and let 
$T: \Ce(K) \to \Ell{1}(L,\mu)$ and $S: \Ce(L) \to \Ell{1}(K,\nu)$ 
be  linear operators such that
\begin{equation}\label{ot.eq:rep-TS}
 \int_L Tf \cdot g \, \ud{\mu} = \int_K f \cdot Sg\, \ud{\nu}
\qquad (f\in \Ce(K),\:g\in \Ce(L)).
\end{equation}
If one of the operators
$T$ and $S$ is order-bounded, then so is the other and  
 \eqref{ot.eq:rep-TS} holds with $T$ and $S$ replaced by 
$\abs{T}$ and $\abs{S}$, respectively. Moreover,
$\mu_T = r_*\nu_S$, where
$r: L \times K \to K \times L$ is the swapping map defined by 
$r(x,y)  = (y,x)$. 
\end{prop}

\begin{proof}
Suppose that $S$ is order-bounded. Then, for $f\in \Ce(K)$ and 
$g\in \Ce(L)$ with $\abs{g}\le 1$,
\[ \Big| \int_L Tf\cdot g\, \ud{\mu} \Big| \le \int_K \abs{f} \cdot \abs{Sg}
\, \ud{\nu}
\le \int_K \abs{f} (\linabs{S}\car)\ud{\nu}.
\]
Hence, $T$ extends to a contraction $\Ell{1}(K, (\linabs{S}\car) \nu) \to
\Ell{1}(L,\mu)$, whence, by Theorem \ref{ot.t:ob-char}, is order-bounded and 
$\abs{T}'\!\mu \le (\linabs{S}\car) \nu$. 
(Recall that the unit ball of $\Ce(L)$ is $\Ell{1}$-dense in the unit ball of
$\Ell{\infty}(L,\mu)$.)

\noindent
In order to prove the first of the two remaining claims, fix $0 \le g\in \Ce(L)$, and
let $f\in \Ce(K)$ and $u\in \Ce(K)$ with $\abs{u}\le 1$. Then 
\begin{align*} 
\bigl| \int_L Tf\cdot (gu)\, \ud{\mu} \bigr|
= \bigl| \int_K f\cdot S(gu)\, \ud{\nu}  \bigr|
\le \int_K \abs{f} \abs{S}\!g\, \ud{\nu}.
\end{align*}
Taking the supremum over all these  functions $u$, we obtain
\[ \int_L \abs{Tf} \cdot g\, \ud{\mu} \le 
\int_K \abs{f} \linabs{S}g\, \ud{\nu}.
\]
This means that $T$ extends to a contraction $T: 
\Ell{1}(K, (\linabs{S}g)\nu) \to \Ell{1}(L, g\mu)$. It follows that
$\abs{T}_g'\!(g\mu) \le (\linabs{S}g) \nu$, where $\abs{T}_g$ denotes
the modulus of $T$ considered as an operator $\Ce(K) \to \Ell{1}(L, g\mu)$. 
However, since $\Ell{1}(L,\mu)$ ``embeds''  onto an ideal of $\Ell{1}(L,g\mu)$,
it follows that $\abs{T}_g = \abs{T}$. Putting things together we obtain
\[ \int_L \linabs{T} f\cdot g\,\ud{\nu} 
=  \int_K f \, \ud\abs{T}_g'\!(g\mu) \le \int_K f \cdot \linabs{S}g\, \ud{\nu}
\]
for $0 \le f\in \Ce(K)$.  The converse inequality holds by symmetry, and 
the last remaining statement is obtained by 
integrating  both measures against functions of the form $f\tensor g$. 
\end{proof}

\noindent
Suppose that $T: \Ce(K) \to \Ell{1}(L,\mu)$ is order-bounded.
Then $\abs{\mu_T} = \mu_{\abs{T}}$ by Theorem \ref{ot.t:rep-mod}, whence
by standard integration theory there is
a $\mu_{\abs{T}}$-almost everywhere unique $\lambda \in \Ell{\infty}(K\times L;
\mu_{\abs{T}})$ with $\abs{\lambda}=1$ almost everywhere and 
\begin{equation}\label{ot.eq:lambda}  
\int_{K \times L} F(x,y) \, \ud{\mu_T} = \int_{K \times L} 
F(x,y) \lambda(x,y) \, \ud{\mu_{\abs{T}}}
\end{equation}
for all $F\in \Ell{1}(K\times L; \mu_{\abs{T}})$.
This leads to the following corollary for situation that $K = L$ and $\mu = \nu$.

\begin{cor}\label{ot.c:rep-sym}
Let $(K,\mu)$ be a compact measure space, let 
$T: \Ce(K) \to \Ell{1}(K,\mu)$ be an order-bounded operator, and let
$\lambda \in \Ell{\infty}(K\times K, \mu_{\abs{T}})$ with 
$\abs{\lambda} =1$ almost everywhere and such that \eqref{ot.eq:lambda} holds
for all $F\in \Ell{1}(K\times L; \mu_{\abs{T}})$. 
Suppose, in addition, that $T$ is symmetric, i.e., $T$ satisfies
\[ \int_K Tf \cdot \conj{g}\, \ud{\mu} =  \int_K f \cdot \conj{Tg}\, \ud{\mu}
\qquad (f,g\in \Ce(K)).
\]
Then  $\abs{T}$ is symmetric, too, and 
\[ \lambda(x,y) = \conj{\lambda(y,x)} \quad \text{for $\mu_{\abs{T}}$-almost all
$(x,y) \in K^2$}.
\]
\end{cor}

\begin{proof}
Note that, by hypothesis,  \eqref{ot.eq:rep-TS} 
holds with  $S = \conj{T}$, whence it holds for
$T$ and $S$ replaced by $\abs{T}$ and $\abs{S} = \abs{T}$, respectively. 
It follows that $\abs{T}$ is symmetric and that $r_*\mu_{\abs{T}}= \mu_{\abs{T}}$.
The last assertion is now straightforward. 
\end{proof}

\medskip
\noindent
The following is the main result of this section.  It has essentially
been proved by  Carbonaro and \Dragicevic\ \cite[p.22/23]{CarbDrag13Pre}.

\begin{thm}[Disintegration]\label{ot.t:disint} 
Let $(K,\mu)$ be a compact measure space, and let $T$ be 
a symmetric absolute contraction on $\Ell{1}(K,\mu)$.
Then
\begin{align*}
 \int_K & (\Id - T)f \cdot g \, \ud{\mu}  = 
\int_K (\Id - M_{\abs{T}\car})f \cdot g \, \ud{\mu} 
\\ & + 
\int_{K \times K}  \int_{\prZ_2} 
\Bigl[ \Id - 
\begin{pmatrix} 0 & \conj{\lambda(x,y)} \\
\lambda(x,y) & 0
\end{pmatrix}\Bigr] 
\begin{pmatrix} f(x) \\ f(y)\end{pmatrix}
\cdot \begin{pmatrix} g(x) \\ g(y)\end{pmatrix}
 \, \ud{\zeta_2} 
\,\ud{\mu_{\abs{T}}}(x,y)
\end{align*}
for all $f\in \Ell{p}(K, \mu)$, $g\in \Ell{q}(K,\mu)$, 
$1\le p \le \infty$. 
\end{thm}

\begin{proof}
We first write  $\Id- T = (\Id - M_{\abs{T}\car}) +  (M_{\abs{T}\car}-T)$ and
then compute
\begin{align*}
\int_K & (M_{\abs{T}\car} - T)f \cdot g\, \ud{\mu}
= 
\int_K (\abs{T}\car)f \cdot g  \,\ud{\mu} - \int_K Tf\cdot g \,\ud{\mu}
 \\ &=
\int_{K^2} \car \tensor fg \, \ud{\mu_{\abs{T}}} - \int_{K^2}
f\tensor g \, \ud{\mu_T}
=
\int_{K^2} \car \tensor fg - (f\tensor g)\lambda \, \, \ud{\mu_{\abs{T}}}. 
\end{align*}
Since $T$ is symmetric and $\abs{\conj{T}}= \abs{T}$, also $\abs{T}$
is symmetric and $\mu_{\abs{T}}$ is a symmetric positive measure. 
Therefore, by a change of variable $(x,y) \mapsto (y,x)$ in the
formula from above,
\begin{align*}
\int_K & (M_{\abs{T}\car} - T)f \cdot g\, \ud{\mu}
=
\int_{K^2} fg \tensor \car - (g\tensor f) \conj{\lambda} \, \, \ud{\mu_{\abs{T}}}. 
\end{align*}
Taking the arithmetic average of this and the previous form we
obtain the claimed formula. 
\end{proof}

\begin{cor}\label{ot.c:disint-subM} 
Let $(K,\mu)$ be a compact measure space, and let $T$ be
a symmetric sub-Markovian operator on $\Ell{1}(K,\mu)$. Then
\begin{align*}
 \int_K  (\Id - T)f \cdot g \, & \ud{\mu}  = 
\int_K (\car - T\car)f \cdot g \, \ud{\mu} 
\\ & + 
\int_{K \times K}  \int_{\prZ_2} 
\begin{pmatrix} 1 & -1  \\
-1 & 1
\end{pmatrix}
\begin{pmatrix} f(x) \\ f(y)\end{pmatrix}
\cdot \begin{pmatrix} g(x) \\ g(y)\end{pmatrix}
 \, \ud{\zeta_2}  \,\ud{\mu_{T}}(x,y)
\end{align*}
for all $f\in \Ell{p}(K, \mu)$, $g\in \Ell{q}(K,\mu)$, 
$1\le p \le \infty$. 
\end{cor}

\section{Proof of the Main Results}\label{s.prf}

Let us return to the proof of the main result, Theorem \ref{mr.t:main}. 
By the reduction steps from Section \ref{s.red}, 
one can suppose from the start that $\prX= (K,\mu)$ is a compact measure
space, $A= \Id - T$ for some symmetric absolute contraction on 
$\Ell{1}(K, \mu)$. In particular, 
the Disintegration Theorem   \ref{ot.t:disint} is applicable. 

Let, as in the hypothesis of Theorem \ref{mr.t:main}, $1\le p <\infty$, 
$d,\:m\in \N$ and $\Eff_j, \Gee_j: K \to \C^d$ be measurable functions
for $1 \le j \le m$. The assertion to prove is:

\smallskip
\noindent
{\em
For all measurable functions $\bff \in \calM(K,\mu; \C^d)$
such that $\Eff_j(\bff)\in \Ell{p}(K, \mu)$ and
$\Gee_j(\bff) \in \Ell{q}(K,\mu)$ for all $1\le j \le m$:
\[ 
\sum_{j=1}^m \re \int_K (\Id - T) \Eff_j(\bff) \cdot \Gee_j(\bff)\,\ud{\mu}
\,\,\ge\,\, 0. 
\]
}%
and we may suppose that this assertion holds when 
$(K,\mu)$ is replaced by $\prZ_2$, and $T$ is replaced by $E_\lambda$ for
each $\lambda\in \torus$. 

\begin{lem}\label{prf.l:scalarcase}
Under the given hypotheses, 
\begin{equation}\label{prf.eq:scalarcase}
 \Re \sum_{j=1}^m \Eff_j(x) \Gee_j(x) \ge 0 \quad \text{for all $x \in \C^d$}.
\end{equation}
\end{lem}

\begin{proof}
Note that the integral inequality is
{\em convex} in $T$, and that it holds trivially for $T= \Id$.
Since it holds for each $T=E_\lambda$, $\lambda \in \torus$,
it also holds for 
$T= \frac{1}{2}E_{1} + \frac{1}{2}E_{-1}= 0$. 
Given $(x_1, \dots, x_d) \in \C^d$, let $f_j := (x_j, x_j)^t \in\calM(\prZ_2)$
and inserting this into the inequality with $T= 0$ on $\prZ_2$ yields the claim.
\end{proof}

\noindent
Suppose now that 
$\bff  \in \calM(K, \mu; \C^d)$ such that 
$\Eff_j(\bff) \in \Ell{p}(K,\mu)$ and $\Gee_j(\bff) 
\in \Ell{q}(K,\mu)$. We can apply the Disintegration Theorem \ref{ot.t:disint} 
and obtain, for each $j= 1, \dots, m$
\begin{align*}
 \int_K & (\Id - T)\Eff_j(\bff) \cdot \Gee_j(\bff) \, \ud{\mu}  = 
\int_K (\car - \abs{T}\car)\Eff_j(\bff)\Gee_j(\bff) \, \ud{\mu}  
\\*  & + 
\int_{K \times K}  \int_{\prZ_2} 
\bigl(\Id - E_{\lambda(x,y)}\bigr) 
\begin{pmatrix} \Eff_j(\bff(x))\\ \Eff_j(\bff(y))\end{pmatrix}
\cdot 
\begin{pmatrix} \Gee_j(\bff(x))\\ \Gee_j(\bff(y))\end{pmatrix}
 \, \ud{\zeta_2} 
\,\ud{\mu_{\abs{T}}}(x,y).
\end{align*}
Now sum over $j$ and take the real part. Finally,  
apply Lemma \ref{prf.l:scalarcase}
for the first summand and the hypothesis over $E_{\lambda(x,y)}$
for the second to conclude that the result has to be $\ge 0$.
Hence, Theorem \ref{mr.t:main} is completely proved.

\medskip
\noindent
The corresponding results for symmetric sub-Markovian and Markovian semigroups
(Theorem \ref{mr.t:main-subM}, Theorem \ref{mr.t:main-M}) are proved similarly.
(Note that by the reduction steps in Section \ref{s.red} one only needs to show
the assertion for the case that $A= \Id - T$ where $T$ is a symmetric absolute contraction 
on a compact measure space $(K,\nu)$, and $T$ is sub-Markovian or Markovian, respectively. 

In the sub-Markovian case (Theorem \ref{mr.t:main-subM}), the hypothesis tells in particular that 
the statement is true for $T=0$ on $\prZ_2$, hence \eqref{prf.eq:scalarcase}
holds. Now apply Corollary \ref{ot.c:disint-subM} and proceed as before.

In the Markovian case, one has $T\car = \car$ and the first summand in the disintegration  
formula of Corollary \ref{ot.c:disint-subM} vanishes. This leads to Theorem \ref{mr.t:main-M}.
(Note that in the Markovian case, \eqref{prf.eq:scalarcase} is not a necessary condition any more.)

\section{Application: The Sector of Analyticity}\label{s.soa}

Let $(S_t)_{t\ge 0}$ be an absolute contraction semigroup over a measure
space $\prX$, and let $1 < p < \infty$. 
As a consequence of the Lumer--Phillips
theorem, the semigroup $(S_t)_{t\ge 0}$ extends
to an analytic contraction semigroup on $\Ell{p}(\prX)$ defined on the sector
\[ \Sigma_\vphi := \{ z\in \C\ohne 0 \st \abs{\arg z} < \vphi\}
\]
(where $0 < \vphi \le  \tfrac{\pi}{2}$) 
if and only if 
\begin{equation}\label{eq.analsgrp}
 \Re \int_\prX \ue^{\pm \vphi \ui} (Af) \cdot \conj{f} \abs{f}^{p{-}2} \, \ge 0
\qquad 
\end{equation}
for all $f\in \dom(A_p)$. For some time it had been an open question
whether, in the case that $(S_t)_t$ is a {\em symmetric} contraction
semigroup, inequality \eqref{eq.analsgrp} must hold for the angle
$\vphi =\vphi_p$, where
\begin{equation}\label{eq.vphi_p}
 \vphi_p := \arccos\abs{1 - \tfrac{2}{p}} = \arctan\frac{2\sqrt{p-1}}{\abs{p-2}}
\end{equation}
for $1 < p < \infty$. Such a result had been first established 
by Bakry \cite{Bakry1989} for a certain subclass of sub-Markovian symmetric 
semigroups and later extended to all sub-Markovian symmetric 
semigroups by Liskevich and Perelmuter \cite{LiskevichPerelmuter1995}.  That proof was subsequently 
improved by Nagel and Voigt \cite{NagelVoigt1996} and in that form
became part of Chapter 3 in Ouhabaz' book \cite{Ouhabaz2005}. The best
general result for all symmetric contraction semigroups had for a
long time been the one by Cowling \cite{Cowling1983}, when Kriegler finally
settled the case with a positive answer in \cite{Kriegler2011}. 
Carbonaro and \Dragicevic\ showed in \cite[Remark 35]{CarbDrag13Pre}
that the optimal angle can be obtained also from their results. 

We shall see in this section that the general symmetric case reduces
to the same scalar inequality as the sub-Markovian case. 
We apply Theorem \ref{mr.t:main} with  $d = m = 1$,
$F(x) = x$ and $G(x) = \ue^{\pm \ui \vphi} \conj{x}\abs{x}^{p-2}$
($G(0)= 0$).  This yields
the inequality
\[ \Re\Bigl( \ue^{\pm \ui \vphi} 
\begin{pmatrix} 1 & -\conj{\lambda} \\ -\lambda & 1\end{pmatrix}
\begin{pmatrix} z \\ w\end{pmatrix} \cdot_{\prZ_2}
\begin{pmatrix} z\abs{z}^{p-2} \\ w \abs{w}^{p-2}\end{pmatrix}
\Bigr) \ge 0
\]
for all choices of $z,\:w\in \C$ and $\lambda \in \torus$.
(Recall that $\cdot_{\prZ_2}$ denotes the sesquilinear inner product
on $\Ell{2}(\prZ_2)$.) 
If we replace $w$ by $\lambda w$ in this inequality, we obtain the
equivalent inequality 
\[ \Re\Bigl( \ue^{\pm \ui \vphi} 
\begin{pmatrix} 1 & -1 \\ -1 & 1\end{pmatrix}
\begin{pmatrix} z \\ w\end{pmatrix} \cdot_{\prZ_2}
\begin{pmatrix} z\abs{z}^{p-2} \\ w \abs{w}^{p-2}\end{pmatrix}
\Bigr) \ge 0.
\]
For $w= 0$ the inequality reduces to $\abs{z}^p \cos\vphi\ge 0$,
which poses no further restriction on $\vphi$. 
For $w \neq 0$ we can replace $z$ by $wz$ and find the equivalent inequality
\[ \Re\Bigl( \ue^{\pm \ui \vphi} 
\begin{pmatrix} 1 & -1 \\ -1 & 1\end{pmatrix}
\begin{pmatrix} z \\ 1\end{pmatrix} \cdot_{\prZ_2}
\begin{pmatrix} z
\abs{z}^{p-2} \\ 1\end{pmatrix}
\Bigr) \ge 0,
\]
i.e., 
\[ \Re \big(\ue^{\pm \ui \vphi} (z-1) ( \conj{z}\abs{z}^{p-2} -1) \big) \ge 0.
\]
 Reformulating this as an inequality between real and imaginary part
and letting $\vphi= \vphi_p$ as above 
reduces to  the inequality (2.1) in \cite{LiskevichPerelmuter1995}
which is proven there. (Actually, our argument
shows that the proof can be simplified  since there is only
one complex variable to deal with.)

\begin{cor}[Kriegler]
Let $-A$ be the generator of a symmetric contraction semigroup
$S = (S_t)_{t\ge 0}$ over some measure space $\prX$, and let $1 < p < \infty$. 
Then $S$ extends to an analytic semigroup of contractions 
on $\Ell{p}(\prX)$ on the sector $\Sigma_{\vphi_p}$.
\end{cor}

\vanish{
This inequality is invariant under replacing $z$ by $\conj{z}$, and
--- up to a positive factor --- by replacing $z$ by $1/z$. Since
for $z\in \R$ and $\abs{z} = 1$ the number $\zeta := (z-1) ( \conj{z}\abs{z}^{p-2} -1)$ is real, 
it suffices to study \eqref{...} for $\abs{z} < 1$ and $\im z  < 0$. In this
case, $\im (z-1) ( \conj{z}\abs{z}^{p-2} -1) = \im z ( \abs{z}^{p-2} -1) > 0$.
The critical angle $\vphi_p$ hence will be the minimal value of all
$\vphi_z \in [0, \upi/2]$, where 
\[ \Re \big(\ue^{\ui \vphi_z} (z-1) ( \conj{z}\abs{z}^{p-2} -1) \big) = 0
\]
for $\abs{z} < 1$ and $\im z < 0$. 
}

\appendix

\section{On Homomorphisms of Probability Spaces}\label{a.embed}

Suppose that $\prX= (X, \Sigma,\mu)$ and $\prX'= (X', \Sigma',\mu')$ are
probability spaces and
\[ \Phi: \Ell{1}(\prX) \to \Ell{1}(\prX')
\]
is a one-preserving isometric lattice 
homomorphism.\footnote{In \cite[Chap.{} 12]{EFHNPre},
this is called a {\em Markov embedding}. It is the functional-analytic analogue
of a {\em factor map} (=homomorphism in the category of
probability spaces) $\prX' \to \prX$.}
 This means that 
$\Phi$ is an isometric embedding for the $\Ell{1}$-norms, $\Phi(\car) = \car$
and $\abs{\Phi f} = \Phi\abs{f}$ for all $f\in \Ell{1}(\prX)$.

The positivity of $\Phi$ implies in particular that
$\Phi(\conj{f}) = \conj{\Phi f}$ for all $f\in \Ell{1}(\prX)$. 
Finally, 
\[ \int_\prX f  = \int_{\prX'} \Phi f
\]
for all $f\in \Ell{1}(\prX)$, since this is true for all $f\ge 0$.

\medskip

\noindent
In this appendix
we show  how to (canonically) extend $\Phi$ to a homomorphic 
(as lattices and $*$-algebras) embedding 
\[ \Phi: \calM(\prX) \to \calM(\prX')
\]
where  $\calM(\prX)$ and $\calM(\prX')$ denote the spaces of 
all measurable $\C$-valued functions modulo almost
everywhere equality on $X$ and  $X'$, respectively.
Note that $\calM(\prX)$ is a complete metric space with respect to the
metric
\[ \ud_\prX(f,g) := \int_{\prX} \frac{\abs{f-g}}{1 + \abs{f-g}}.
\]
The following lemma is the key property.

\begin{lem}\label{app.l:alghom}
In the situation from above, $\Phi$ restricts to an  embedding
of $C^*$-algebras $\Phi: 
\Ell{\infty}(\prX) \to  \Ell{\infty}(\prX')$.
Moreover, for any $f\in \Ell{1}(\prX)$, 
\[ \mu\set{ \abs{f} > 0} = \mu'\set{\abs{\Phi f}  > 0}
\]
In particular, $\set{f = 0}$ is a $\mu$-null set if and only if 
$\set{\Phi f = 0}$ is a $\mu'$-null set.  
\end{lem}

\begin{proof}
It is clear that $\Phi$ restricts to a one-preserving isometric lattice
homomorphism between the respective $\Ell{\infty}$-spaces. So only the
multiplicativity $\Phi(fg)= (\Phi f) (\Phi g)$ is to be shown. 
This is well-known, see e.g. \cite[Chap.{} 7]{EFHNPre},
but we repeat the argument for the convenience of the reader. By bilinearity,
it suffices to consider $f,g \ge 0$. Then, by polarization, it
suffices to consider $f=g$, which reduces the problem to establish that
$\Phi(f^2) = (\Phi f)^2$. Now, for any $x\ge 0$, $x^2 = \sup_{t\ge 0} 2tx - t^2$. 
Hence, $f^2 = \sup_{t\ge 0} 2tf - t^2 \car$ in the Banach lattice sense. But
$\Phi$ is a lattice homomorphism and $\Phi\car = \car$, whence
\[ \Phi(f^2) = \Phi\bigl(\sup_{t\ge 0} 2tf - t^2 \car \bigr) = 
\sup_{t\ge 0} 2t(\Phi f) - t^2 \car = (\Phi f)^2.
\]
The remaining statement follows from: 
\begin{align*}
\mu\set{ \abs{f} > 0} & = \lim_{n \to \infty} \int_\prX (n \abs{f} \Inf \car) 
=
\lim_{n \to \infty} \int_{\prX'} \Phi( n \abs{f} \Inf \car) \\ 
& = \lim_{n \to \infty} \int_{\prX'} n \abs{\Phi f} \Inf \car = 
\mu'\set{\abs{\Phi f} > 0}.\qedhere
\end{align*}
\end{proof}

\noindent
Let $f\in \calM(\prX)$. Then the function 
$e := \frac{1}{1 + \abs{f}}$
has the property that $e, ef \in \Ell{\infty}(\prX)$. Moreover, 
by Lemma \ref{app.l:alghom}, $\set{\Phi e = 0}$ is   $\mu'$-null set. Hence,
$\Phi e $ is an invertible element in the algebra $\calM(\prX')$, 
and we can define
\[ \widehat{\Phi}f := \frac{\Phi(ef)}{\Phi e} \in\calM(\prX').
\]

\begin{lem}\label{app.l:Phihut}
The so-defined mapping $\widehat{\Phi}: \calM(\prX) \to \calM(\prX')$ 
has the following properties: 
\begin{aufzi}
\item $\widehat{\Phi}$ is an extension of $\Phi$.
\item $\widehat{\Phi}$ is a unital $*$-algebra and lattice homomorphism.
\item $\displaystyle \int_{\prX'} \widehat{\Phi} f  = 
\int_\prX f $ whenever
$0 \le f\in \calM(\prX)$.
\item $\widehat{\Phi}$ is an isometry with respect to the canonical metrics 
$\ud_{\prX}$ and $\ud_{\prX'}$. 

\item If $\Phi$ is bijective then so is $\widehat{\Phi}$.

\item The mapping  $\widehat{\Phi}: \calM(\prX) \to \calM(\prX')$  
is uniquely determined by the property that it extends $\Phi$ and 
it is multiplicative, i.e., satisfies $\widehat{\Phi}(fg) = \widehat{\Phi}f \cdot \widehat{\Phi}g$
for all $f,\:g \in \calM(\prX)$.
\end{aufzi}
\end{lem}

\begin{proof}
a) and b)\ This is straightforward and left to the reader.\\
c)\ By the monotone convergence theorem,
\begin{align*}
  \int_\prX f  &= \sup_{n\in \N} \int_\prX (f\Inf n\car)\, 
= \sup_{n\in \N} \int_\prX \Phi(f\Inf n\car)\, 
=
\sup_{n\in \N} \int_{\prX'} \widehat{\Phi}(f\Inf n\car)
\\ &  = \sup_{n\in \N} \int_{\prX'} (\widehat{\Phi} f \Inf n\car)
= \int_{\prX'} \widehat{\Phi} f. 
\end{align*}
d)\ Follows from b) and c).\\
e)\  Suppose that $\Ell{\infty}(\prX') \subseteq \ran(\Phi)$ and
let $g\in \calM(\prX')$ be arbitrary.
Then, by Lemma \ref{app.l:alghom}, there are 
$e,h\in \Ell{\infty}(\prX)$ such that 
\[ \Phi e = \frac{1}{1 + \abs{g}}\quad \text{and}
\quad \Phi h = \frac{g}{1 + \abs{g}} = g\: \Phi e.
\]
Again by Lemma \ref{app.l:alghom}, $\mu\set{e = 0}= 0$, whence we can define
$f:= \frac{h}{e} \in \calM(\prX)$. It follows that $\Phi f= g$.\\
f) 
Suppose that 
$\Psi: \calM(\prX) \to \calM(\prX')$ is multiplicative and extends $\Phi$. 
Let  $f \in \calM(\prX)$ and define $e := \frac{1}{1 + \abs{f}}$ as before. Then 
$f,\: ef\in \Ell{\infty}(\prX)$ and hence 
\[ \Phi e \cdot \Psi f =  \Psi e \cdot \Psi f = \Psi(ef) = \Phi(ef).
\]
Since $\Phi e$ is an invertible element in $\calM(\prX')$ (as seen above), it follows that 
\[ \Psi f = \frac{\Phi(ef)}{\Phi e} = \widehat{\Phi}f
\]
as claimed. 
\end{proof}

\noindent
By abuse of notation, we write $\Phi$ again instead of $\widehat{\Phi}$.
It is clear that $\Phi$ allows a further extension to 
$\C^d$-valued functions by 
\[ \Phi(\bff) = \Phi(f_1, \dots, f_d) := (\Phi f_1, \dots, \Phi f_d)
\quad \text{for $\bff = (f_1, \dots, f_d) \in \calM(\prX;\C^d)$}.
\] 
Now we are well-prepared for the final result of this appendix.

\begin{thm}\label{embed.t:fc-gen}
Let $\prX$ and  $\prX'$ be probability
spaces, and let $\Phi: \Ell{1}(\prX) \to \Ell{1}(\prX')$
be a one-preserving isometric lattice isomorphism, with its
canonical extension $\Phi: \calM(\prX; \C^d) \to \calM(\prX';\C^d)$, $d\in \N$.
Then 
\begin{equation}\label{app.eq:fc}
 \Phi\bigl( F(\bff)\bigr)
= F(\Phi \bff)
\qquad \text{almost everywhere}
\end{equation}
for every Borel measurable function $F: \C^d \to \C$ and 
every $\bff  \in \calM(\prX;\C^d)$.
\end{thm}

\begin{proof}
By linearity we may suppose that $F \ge 0$. Next, by approximating
$F \Inf n \car \nearrow F$, we may suppose that $F$ is bounded. 
Then $F$ is a uniform limit of positive simple functions, whence
we may suppose without loss of generality that $F = \car_B$, where
$B$ is a Borel set in $\C^d$. In this case, \eqref{app.eq:fc} becomes
\[ \Phi\bigl( \car_{\set{(f_1, \dots, f_d) \in B}} \bigr)
= \car_{\set{(\Phi f_1, \dots, \Phi f_d) \in B}} \qquad \text{almost everywhere}.
\]
Let $\calB$ be the set of all Borel subsets of $\C^d$ that 
satisfy this. Then $\calB$ is a Dynkin system, so it suffices to
show that each rectangle is contained in $\calB$. 
Since $\Phi$ is multiplicative,
this reduces the case to $d=1$, $f$ is real valued and $B = (a, b]$. 
Now $\set{ a < f \le b} = \set{a < f } \cap \set{b < f}^c$, which
reduces the situation to $B = (a, \infty)$. Now
\[  \car_{\set{a < f}} = \text{$\Ell{1}$-}\lim_{n \to \infty}
n (f - a\car)^+ \Inf \car,
\]
and applying $\Phi$ concludes the proof. 
\end{proof}

\begin{rems}
\begin{aufziii}
\item As a consequence of Theorem \ref{embed.t:fc-gen}, 
$\Phi \abs{f}^p = \abs{\Phi f}^p$ for any $f\in \calM(\prX)$ and 
$p> 0$, so $\Phi$ restricts to an isometric isomorphism of $\Ell{p}$-spaces
for each $p > 0$. 
\item The extension of the original $\Ell{1}$-isomorphism $\Phi$
to $\calM(\prX)$ is uniquely determined by the requirement that 
$\Phi$ is continuous for the metrics $\ud_\prX$ and $\ud_{\prX'}$.

\item One can extend $\Phi$ to a lattice homomorphism 
\[ \Phi: \calM(\prX;[0, \infty]) \to \calM(\prX';[0,\infty])
\]
by defining $\Phi f := \tau^{-1} \nach \Phi(\tau \nach f)$, 
where $\tau: [0, \infty] \to [0,1]$ is any order-preserving bijection. 
Using this one can then show that $\Phi$ maps almost everywhere convergent
sequences to almost everywhere convergent sequences.
\end{aufziii}
\end{rems}

\noindent
{\bf Acknowledgement:}\ 
I am very grateful to  Tom ter Elst (Auckland) for his kind 
invitation in March/April 2014, and the pleasure of 
studying together the paper \cite{CarbDrag13Pre}. Furthermore, 
I thank Hendrik Vogt for some stimulating discussions and
the anonymous referee for his careful reading of the 
manuscript and some very useful remarks.

\bibliographystyle{acm} 

\def\cprime{$'$}

\end{document}